\documentclass[10pt]{article}
\usepackage[margin=1in]{geometry}

\usepackage[utf8]{inputenc} 
\usepackage{graphicx}
\usepackage[T1]{fontenc}    
\usepackage{hyperref}       
\usepackage{url}            
\usepackage{booktabs}       
\usepackage{amsfonts}       
\usepackage{nicefrac}       
\usepackage{microtype}      
\usepackage{xcolor}         
\usepackage{amsmath,amssymb,amsthm}
\usepackage{algorithm,algorithmic}
\usepackage{subcaption}
\usepackage{dsfont}
\usepackage{multirow}
\usepackage{bm}
\usepackage{wrapfig}

\title{Log-Averaged Mirror Prox for Fast, Large-Scale Optimal Transport in Linear Space}

\newcommand{\R}{\mathbb{R}} 

\renewcommand{\O}{\mathcal{O}}
\newcommand{\tO}{\widetilde{\O}}

\newcommand{\DHa}{\D_{H_c^\alpha}}

\newcommand{\softmin}{\operatorname{smin}}

\DeclareMathOperator*{\argmax}{argmax}
\DeclareMathOperator*{\argmin}{argmin}

\DeclareMathOperator*{\Argmin}{Argmin}

\newcommand{\inner}[2]{\left\langle #1,#2\right\rangle}

\makeatletter
\newcommand\footnoteref[1]{\protected@xdef\@thefnmark{\ref{#1}}\@footnotemark}
\makeatother

\newcommand{\Rp}{\R_{+}}

\newcommand{\coup}{\Pi}
\newcommand{\D}{\operatorname{D}}

\newcommand{\LSE}{\mathrm{LSE}}

\newcommand{\diag}[1]{\mathcal{D}_{#1}}

\theoremstyle{plain}
\newtheorem{theorem}{Theorem}[section]
\newtheorem{proposition}[theorem]{Proposition}
\newtheorem{lemma}[theorem]{Lemma}

\theoremstyle{definition}

\theoremstyle{remark}

\newcommand{\KL}{\mathrm{D}_\mathrm{KL}}
\newcommand{\one}{\mathbf{1}}
\newcommand{\zero}{\mathbf{0}}
\newcommand{\col}{\bm{\mathrm{c}}}
\newcommand{\row}{\bm{\mathrm{r}}}
\author{
  Matthew X. Burns\thanks{
        Department of Electrical and Computer Engineering, University of Rochester, Rochester, NY 14627 (email: {\tt mburns13@ur.rochester.edu}).}
  \and
  Jiaming Liang\thanks{
        Goergen Institute for Data Science and Artificial Intelligence (GIDS-AI) and Department of Computer Science, University of Rochester, Rochester, NY 14620 (email: {\tt jiaming.liang@rochester.edu}). This work was partially supported by GIDS-AI seed funding and AFOSR grant FA9550-25-1-0182.}
}
\date{May 11, 2026 (1$^\text{st}$ revision November 13, 2026)}

\begin{document}

\maketitle

\begin{abstract}
 We propose Log-Averaged Mirror Prox (LAMP), a linear-space primal-dual method for large-scale optimal transport.
LAMP implements primal mirror prox updates by tracking an averaged dual sequence, reducing storage complexity from $\O(nm)$ to $\O(n+m)$ while preserving dense, GPU-friendly reductions.
Consequently, LAMP preserves the last-iterate $\tO(\|C\|_\infty nm\varepsilon^{-1})$ arithmetic complexity of conservatively parameterized primal-dual mirror prox. We further analyze LAMP as a direct optimal transport solver in a more performant parameter regime, providing a last-iterate sub-optimality certificate dependent on infeasibility and an explicit $\mathcal O(\|C\|_\infty\log (m)/t)$ term.  Moreover, we give a computable sufficient condition for {ergodic} convergence.
Numerical experiments with an optimized CUDA implementation show that LAMP outperforms first-order baselines in several high-accuracy (entropic) optimal transport problems. LAMP is further shown to scale up to problems with $n=m=2^{18}$ marginal supports, which were previously beyond the reach of primal-dual first-order methods.
\end{abstract}

\section{Introduction}\label{sec:intro}
Given probability mass functions $r\in\R^n$, $c\in\R^m$ and a cost matrix $C\in\R_+^{n\times m}$, the (discrete) optimal transport (OT) problem is the linear program
\begin{equation}\label{prob:primal_ot}
\min_{X\in\coup(r,c)}\inner{C}{X},
\end{equation}
where $\coup(r,c)$ is the set of couplings between $r$ and $c$. Variations of OT have found applications in generative modeling~\cite{arjovskyWassersteinGAN2017}, computational science~\cite{levyFastSemidiscreteOptimal2021,huizingOptimalTransportImproves2022}, robotics~\cite{yaoVisualTrackingUsing2018}, domain adaptation~\cite{courtyOptimalTransportDomain2017}, and economics~\cite{galichonOptimalTransportMethods2016};  underscoring the need for computationally efficient, large-scale OT. Accordingly, OT has become an increasingly popular area of research at the intersection of first-order optimization and high-performance parallel computing. Following the seminal work of~\cite{cuturiSinkhornDistancesLightspeed2013}, numerous (entropic) OT-focused algorithms have been proposed that target highly parallel, GPU-amenable subroutines. The standard Sinkhorn algorithm~\cite{sinkhornConcerningNonnegativeMatrices1967,cuturiSinkhornDistancesLightspeed2013} computes a solution with $\varepsilon$-additive error in $\tO(\|C\|_\infty^2 nm\varepsilon^{-2})$ operations~\cite{dvurechenskyComputationalOptimalTransport2018}, where $\|C\|_\infty=\max_{ij}C_{ij}$. Primal-dual methods have been proposed to improve the complexity to $\tO(\|C\|_\infty nm(n+m)^{1/2}\varepsilon^{-1})$~\cite{dvurechenskyComputationalOptimalTransport2018,guminovCombinationAlternatingMinimization2021,linEfficiencyEntropicRegularized2022} or  $\tO(\|C\|_\infty nm\varepsilon^{-1})$~\cite{jambulapatiDirectTildelbraceOrbrace2019a,maiFastAccurateSplitting2021a,luoImprovedRateFirst2023} (see Appendix~\ref{app:review} and Table~\ref{tab:summary} for further literature review). However, existing primal-dual methods require primal averaging, which either necessitates $\mathcal{O}(nm)$ storage or GPU-unfriendly recovery subroutines dependent on sparse access patterns and pointer-based data structures~\cite{assadiSemiStreamingBipartiteMatching2022}. In a recent development, entropy-regularized Primal-Dual Mirror Prox (PDMP) methods by~\cite{cen2024fast,liFastComputationOptimal2025} achieve $\tO(\|C\|_\infty nm\varepsilon^{-1})$ last-iterate arithmetic complexity for computing an $\varepsilon$-additive solution: requiring no postprocessing or ergodic averaging. However, as formulated, these methods still require $\mathcal{O}(nm)$ storage, which is prohibitive for large-scale OT applications. In many cases, entries of the cost matrix $C$ can be computed on-the-fly using a distance kernel with $\mathcal{O}(n+m)$ storage, making explicit representation of $X$ the dominant storage cost. In this work, we provide the following contributions:
\begin{enumerate}
    \item We develop Log-Averaged Mirror Prox (LAMP), a dual-only variant of PDMP that works entirely in the dual space. LAMP preserves the theoretical guarantees of PDMP while reducing storage complexity from $\O(nm)$ to  $\mathcal{O}(n+m)$, making the method scalable to large-scale OT instances. Unlike the semi-streaming dual-extrapolation approach of~\cite{assadiSemiStreamingBipartiteMatching2022}, LAMP avoids sparse, pointer-based recovery and maps directly to dense GPU reductions.
    \item We further analyze LAMP as a direct OT method in the practical parameter regime, where the entropic regularization parameter is set to zero.
    In particular, we prove a last-iterate objective-error bound controlled by marginal infeasibility and an explicit $\mathcal O(\|C\|_\infty\log (m)/t)$ term, as well as a computable sufficient condition for ergodic convergence. 
    \item Supported by numerical experiments, we demonstrate that LAMP outperforms comparable GPU-accelerated first-order solvers in several benchmark and real-world OT problems. Our LAMP implementation is publicly available and shown to scale to problems with $n=m=2^{18}$ marginal supports, and we further show our CUDA-accelerated codebase outperforms the baseline PyKeOps library in dense reduction operations.
\end{enumerate}

Proofs, further discussion of related works, and experimental details can be found in the appendices.  

\section{Preliminaries}
Throughout, $\R$ denotes the real numbers and $\Rp$ are the non-negative reals. By $\Delta^n$, we mean the $n$-dimensional unit simplex. Similarly, $\{\Delta^m\}^n$ denotes the set of $n\times m$ row-stochastic matrices (non-negative matrices whose rows sum to one). For a vector $x\in\R^n$, $\diag{x}\in\R^{n\times n}$ is the matrix with $x$ along the diagonal. By $x\in[-1,1]^m$, we mean for all $1\leq i\leq m$, $-1\leq x_{i}\leq 1$.

Given a closed, proper, convex and differentiable function $\omega:\R^n\to(-\infty,\infty]$, we define its Bregman divergence as
\begin{equation*}
    \D_\omega(x\|y)=\omega(x)-\omega(y)-\inner{\nabla \omega(y)}{x-y}.
\end{equation*}

Unless otherwise indicated, $\|\cdot\|$ is the standard Euclidean vector norm, and $\|\cdot\|_1$ and $\|\cdot\|_\infty$ are the $\ell_1$ and $\ell_\infty$ norms. A norm applied to a matrix is meant in an entry-wise manner, e.g., $\|C\|_\infty=\max_{ij}|C_{ij}|$. We denote the Euclidean inner product by $\inner{x}{y}=x^\top y$. When applied to matrices or vectors, functions $\log$, $\exp$, and $\tanh$ are to be interpreted in an entry-wise fashion. We define the entropy of a distribution $\gamma\in\Delta^n$ as $H(\gamma)=-\inner{\gamma}{\log \gamma}$, with the convention that $0\log 0=0$, and the LogSumExp function as $\LSE(x)=\log\left[\sum_{j=1}^m\exp(x_j)\right]$, where $x\in\R^m$. For convenience, we denote the Bregman divergence of the negative entropy as the KL-divergence $\KL(X\|Y)=\D_{-H}(X\|Y)$. For a matrix $A\in \R^{n\times m}$, $\row(A)\in\R^n$ is the row-wise sum and $\col(A)\in\R^m$ is the column-wise sum. We denote $\one_m$ as the all-ones vector in $\R^m$ and $\zero_m$ as the all-zeros vector in $\R^m$. For a matrix $C\in\R^{n\times m}$, we denote the $i^{\text{th}}$ row by $C_{i:}\in\R^m$.

\subsection{Optimal Transport}
For $\varepsilon>0$, we say that $X\in\coup(r,c)$ is an $\varepsilon$-solution to the primal OT problem~\eqref{prob:primal_ot} if
\[\inner{C}{X}-\min_{X\in\coup(r,c)}\inner{C}{X}\leq \varepsilon.\]
As formulated, the OT problem is an $n\cdot m$ dimensional linear program with $m+n$ constraints. Interior point methods exploiting the low-dimensional constraints can solve~\eqref{prob:primal_ot} to high-precision in $\tO((n+m)^{5/2})$ arithmetic operations~\cite{leePathFindingMethods2014}, however the underlying subroutines are not amenable to parallelization, and therefore cannot leverage the increasing availability of GPUs.

The computational challenges of OT have led to the wide adoption of entropic optimal transport (EOT)~\cite{cuturiSinkhornDistancesLightspeed2013,nutzIntroductionEntropicOptimal}. EOT augments the OT objective function with a negative entropy term, making the objective $\eta $-strongly convex with respect to the $\ell_1$ and $\ell_2$ norms,
\begin{equation}
    \min_{X\in\coup(r,c)}\{\inner{C}{X}-\eta  H(X)\}\label{prob:primal_eot}.
\end{equation}

The predominant method for solving~\eqref{prob:primal_eot} is the Sinkhorn-Knopp matrix-scaling algorithm (or just ``Sinkhorn'' for brevity). Sinkhorn is simple to analyze, performs well for large $\eta$ (approximately $\geq 10^{-3}$), and is highly parallelizable. Ever since the popularization of Sinkhorn in~\cite{cuturiSinkhornDistancesLightspeed2013}, numerous works have analyzed~\cite{altschulerNearlinearTimeApproximation2017,dvurechenskyComputationalOptimalTransport2018,ghosalConvergenceRateSinkhorns2025c} or extended~\cite{benamouIterativeBregmanProjections2015,linEfficiencyEntropicRegularized2022} the basic approach. As proven in~\cite{dvurechenskyComputationalOptimalTransport2018}, Sinkhorn has a computational complexity of $\tO(\|C\|_\infty^2 nm\varepsilon^{-2})$ for finding an $\varepsilon$-solution to~\eqref{prob:primal_ot}.
This is unacceptably slow for high-accuracy (small $\varepsilon$) OT, motivating further approaches for EOT based on accelerated gradient descent~\cite{dvurechenskyComputationalOptimalTransport2018}, mirror descent~\cite{linEfficiencyEntropicRegularized2022}, and saddle-point methods~\cite{jambulapatiDirectTildelbraceOrbrace2019a,chambolleAcceleratedBregmanPrimalDual2022}.



\subsection{Penalty and Saddle-Point Formulations}\label{subsec:saddle}
An alternative penalty formulation of OT/EOT investigated by~\cite{jambulapatiDirectTildelbraceOrbrace2019a} and~\cite{liFastComputationOptimal2025} is the problem
\begin{equation}
    \min_{p\in\{\Delta^m\}^n}\{P^{\eta }(\diag{r}p)
    :=\inner{C}{\diag{r}p}+2\|C\|_\infty\|\col_r(p)-c\|_1-\eta  H_r(p)\},\label{prob:l1_ot_primal}
\end{equation}
where $\col_r(p):=\col(\diag{r}p)$, $H_r(p):=H(\diag{r}p)$, and $\eta \geq 0$ is the entropic regularization coefficient. 
There are two primary differences between~\eqref{prob:primal_ot}/\eqref{prob:primal_eot} and~\eqref{prob:l1_ot_primal}. First, we reparameterize the primal variables $X=\diag{r}p$ from the matrix simplex $X\in \Delta^{n\times m}$ to the row-stochastic matrix $p\in\{\Delta^{m}\}^n$. One can easily show that $\diag{r}p\in \Delta^{n\times m}$ with row marginal $\row(\diag{r}p)=r$, therefore the row parameterization directly enforces the row marginal constraint. The reparameterization is not unique to this formulation, and has been used in other (E)OT algorithms~\cite{chambolleAcceleratedBregmanPrimalDual2022}. Second, we replace the constraint $\col_r(p)=c$ with an $\ell_1$ penalty term, making the problem unconstrained but nonsmooth. 

From~\cite[Lemma 2.3]{jambulapatiDirectTildelbraceOrbrace2019a}, an optimizer to problem~\eqref{prob:l1_ot_primal} with $\eta =0$ is reducible to an optimizer to~\eqref{prob:primal_ot} in $\mathcal{O}(nm)$ operations. For completeness, we extend the equivalence result to the case where $\eta >0$ in Appendix~\ref{appdx:deferred} (see Lemma~\ref{lem:l1_prob_equiv}), however our primary target is the unregularized problem~\eqref{prob:primal_ot}.

To practically solve~\eqref{prob:l1_ot_primal}, we dualize the $\ell_1$ penalty with the identity $\|x\|_1=\max_{y\in[-1,1]^m}\inner{y}{x}$ for $x\in\R^m$, obtaining the primal-dual formulation
\begin{align}\label{prob:spp}
    \min_{p\in\{\Delta^m\}^n}&\max_{\theta\in[-1,1]^m}\{K^{\eta }(\diag{r}p,\theta):=\inner{C}{\diag{r}p}  + 2\|C\|_\infty\inner{\theta}{\col_r(p) -c}- \eta  H_r(p)\}.
\end{align}
 Since~\eqref{prob:spp} is convex-concave over compact domains, Sion's minimax theorem~\cite{sionGeneralMinimaxTheorems1958} permits us to interchange the $\min$ and $\max$, and solve~\eqref{prob:spp} as a saddle-point problem.


\section{Log-Averaged Mirror Prox}\label{sec:lamp}
A popular method for solving saddle-point problems over simple, compact sets is mirror prox~\cite{nemirovskiProxMethodRateConvergence2006}, which extends classical extragradient methods to non-Euclidean domains. In this section, we show that recently proposed PDMP~\cite{cen2024fast,liFastComputationOptimal2025} with last-iterate guarantees can be implemented in $\mathcal{O}(n+m)$ storage by connecting primal mirror descent to dual log-averaging.

\subsection{Mirror Descent as Dual Log-Averaging}\label{sec:mirror_maps}

First, we define the Bregman divergence $\D_{H_c^\alpha}(\theta^{a}\|\theta^{b})$ where $H_c^\alpha(\cdot)$ is the negative dual entropy
\begin{align*}
    H_c^\alpha(\nu)=\sum_{j=1}^mc^\alpha_j\biggl(\frac{1+\nu_{j}}{2}\ln\left[\frac{1+\nu_j}{2}\right]+\frac{1-\nu_{j}}{2}\ln\left[\frac{1-\nu_j}{2}\right]\biggr),
\end{align*}
where $c^\alpha:=c+(\alpha/m)\one_m$.
The $\KL$ and $\DHa$ divergences will serve as the movement limiting potentials for the primal and dual steps, respectively, and therefore play a key role in our mirror prox definitions.

Let $F_{\theta}^{\eta }(\cdot)=K^{\eta }(\diag{r}\cdot,\theta)$ be the primal function defined by the saddle objective in \eqref{prob:spp} with the dual variables fixed at $\theta$, and similarly define $G_{p}^{\eta }(\cdot)=K^{\eta }(\diag{r}p,\cdot)$. We then define the primal and dual mirror maps as
\begin{equation}\label{def:mirror_maps}
\begin{split}
    \mathcal{M}_{\tau}^{\eta }(p^0;\theta)&=\argmin_{p\in\{\Delta^m\}^n}\{\tau\inner{\nabla F_\theta^{\eta }(p^0)}{p}+\KL(\diag{r}p\|\diag{r}p^0)\}, \\
   \mathcal{M}_{\tau}^{\eta }(\theta^0;p)&=\argmax_{\theta\in\R^m}\{\tau\inner{\nabla G_{p}^{\eta }(\theta^0)}{\theta}-\D_{H_c^\alpha}(\theta\|\theta^0)\},
   \end{split}
\end{equation}
which we can show have closed-form solutions (see Appendix~\ref{appdx:deferred})
\begin{align}
    \nonumber\mathcal{M}_{\tau}^{\eta }(p^0;\theta)_{ij}=&\frac{(p^0)^{1-\tau\eta }_{ij}}{Z_i}\exp[-\tau(C_{ij}+2\|C\|_{\infty}\theta_j)], \nonumber \\
    \mathcal{M}_{\tau}^{\eta }(\theta^0;p)_j=&\tanh\biggl[\frac{2\tau\|C\|_\infty}{c^\alpha_j}(\col_r(p)_j-c_j)+\frac{1}{2}\log\frac{1+\theta^0_j}{1-\theta^0_j}\biggr]. 
    \label{def:closed_form_opt}
    \end{align}
    Here $Z_i$ is the normalization constant to ensure $\mathcal{M}_\tau^{\eta }(p^0;\theta)\in\{\Delta^{m}\}^n$.
    Additionally, we define the $\tanh$ clipping function $\operatorname{tclip}(\theta, \beta)$, which performs coordinate-wise clipping to the subset $[-\tanh(\beta/2), \tanh(\beta/2)]^m\subseteq[-1,1]^m$ for some $\beta > 0$. As shown in~\cite{liFastComputationOptimal2025}, dual clipping is theoretically useful and substantially improves empirical performance. 
    
    Finally, for $\theta\in[-1,1]^m$ and $\eta>0$, we define the dual-to-primal map $p^\eta(\theta)$ as
\begin{equation}\label{def:dual_to_primal}
    p^\eta(\theta)=\argmin_{p\in\{\Delta^m\}^n}\{F_\theta^{\eta }(p)\}= \diag{Z}^{-1} \exp[-\eta^{-1}(C+2\|C\|_\infty\one_n\theta^\top)],
\end{equation}
where $\diag{Z}^{-1}$ enforces row normalization.

With the main ingredients in place, we prove a simple connection between the primal mirror map and the dual variable $\theta$. The following lemma is the key observation that enables us to reduce PDMP from $\mathcal{O}(nm)$ space to $\mathcal{O}(n+m)$ space in the following subsection by replacing the primal iterates with a weighted average in the dual. The proof is deferred to Appendix~\ref{appdx:deferred}.
\begin{lemma}\label{lem:primal_md_eq_dual_avg}
    Let $\eta\geq 0$, $\gamma>0$, and $\tau>0$ satisfy $\tau\eta\leq 1$. Then let $\theta^a$, $\theta^b\in [-1,1]^m$. Define $\eta'=\gamma / (1+\tau(\gamma-\eta ))$ and $\theta'=\theta^a + \tau\eta'(\theta^b-\theta^a)$. Then, we have the equivalence
    \begin{equation}\label{eq:mirror_map_id}
        p^{\eta'}(\theta')=\mathcal{M}^\eta_{\tau}(p^\gamma(\theta^a);\theta^b).
    \end{equation}
\end{lemma}


\subsection{From PDMP to LAMP}\label{ssec:algorithm}
Using the $\mathcal{M}^{\eta }_\tau(\cdot;p)$ and $\mathcal{M}^{\eta }_\tau(\cdot;\theta)$ primitives, ~\cite{liFastComputationOptimal2025} proposed a PDMP method with last-iterate guarantees. 
Algorithm~\ref{alg:pdmp} shows PDMP, where Round~\cite{altschulerNearlinearTimeApproximation2017} is a standard OT subroutine, which returns a nearby, feasible transport plan in $\mathcal{O}(nm)$ operations (see Appendix~\ref{appdx:deferred}). 
Since PDMP is composed entirely of mirror map operations, Lemma~\ref{lem:primal_md_eq_dual_avg} implies that we can recover the primal sequences $\{p^t\}$ and $\{\bar p^t\}$ using auxiliary dual sequences $\{\nu^t\}$ and $\{\bar \nu^t\}$. The resulting LAMP method is given in Algorithm~\ref{alg:lamp}.
\begin{minipage}[t]{0.49\linewidth}
\begin{algorithm}[H]\caption{Primal-Dual Mirror Prox}\label{alg:pdmp}
    \begin{algorithmic}
        \REQUIRE $C\in\R_{+}^{n\times m}$, $r\in\Delta^n$, $c\in\Delta^m$, $\alpha,\,\beta>0$, $\tau_1,\,\tau_2>0$, $\eta \geq 0$, $T\in\mathbb{N}>0$, set $p^0=(1/m)^{n\times m}$, $\theta^0=\zero_m$, $c^\alpha=c + \alpha m^{-1}\one_m$.
        \FOR{$t=0$ to $T-1$}
            \STATE \textbf{Step 1)} Compute
            \begin{align}
                \bar{p}^{t+1}=&\mathcal{M}_{\tau_1}^{\eta }(p^t;\theta^t)\label{step:pdmp_midpoint}\\ \nonumber\bar{\theta}^{t+1}=&\mathcal{M}_{\tau_2}^{\eta }(\theta^t;p^t)\\
                p^{t+1}=&\mathcal{M}_{\tau_1}^{\eta }(p^t;\bar\theta^{t+1}) \label{step:pdmp_main}\\
                \nonumber\hat\theta^{t+1}=&\mathcal{M}_{\tau_2}^{\eta }(\theta^{t};\bar p^{t+1})\\
                \theta^{t+1}=&\operatorname{tclip}(\hat\theta^{t+1},\beta)\label{eq:pdmp_theta_clipping}
            \end{align}
        
            \vspace{-0.5em}
        \ENDFOR
        \STATE \textbf{Step 2)} \textbf{return} $\operatorname{Round}(\diag{r}p^{T}, r, c)$.
    \end{algorithmic}
\end{algorithm}
\end{minipage}
\hfill
\begin{minipage}[t]{0.49\linewidth}
\begin{algorithm}[H]\caption{Log-Averaged Mirror Prox}\label{alg:lamp}
    \begin{algorithmic}
        \REQUIRE $C\in \Rp^{n\times m}$, $r\in\Delta^n$, $c\in\Delta^m$, $\alpha,\,\beta>0,\,\tau_1,\,\tau_2 > 0$,  $\eta \geq 0$,  $T\in\mathbb{N}$, set $\theta^0=\nu^0=\zero_m$, $c^\alpha=c + \alpha m^{-1}\one_m$, $\eta_0=\infty$.
        \FOR{$t= 0$ to $T-1$}
            \STATE \textbf{Step 1)} Compute           
             \begin{align}
                \eta_{t+1}=&\eta_t/(1+\tau_1(\eta_t-\eta ))\label{step:eta_update}\\
                \bar{\nu}^{t+1}=&\nu^t+\tau_1\eta_{t+1}(\theta^t-\nu^t)\label{step:lamp_mid_nu}\\
                \bar\theta^{t+1}=&\mathcal{M}_{\tau_2}^{\eta }(\theta^{t};p^{\eta_t}(\nu^{t}))\label{step:lamp_mid_theta}\\
                {\nu}^{t+1}=&\nu^t+\tau_1\eta_{t+1}(\bar{\theta}^{t+1}-\nu^t)\label{step:lamp_main_nu}\\
            \hat\theta^{t+1}=&{\mathcal{M}}_{\tau_2}^{\eta }(\theta^{t};p^{\eta_{t+1}}(\bar\nu^{t+1}))\label{step:lamp_main_theta}\\
            \nonumber\theta^{t+1}=&\operatorname{tclip}(\hat\theta^{t+1},\beta)
            \end{align}
            

        \ENDFOR
        \STATE \textbf{Step 2)} \textbf{return} $\operatorname{Round}(\diag{r}p^{\eta_{T}}(\nu^{T}), r, c)$.
    \end{algorithmic}
\end{algorithm}
\end{minipage}



The following proposition formalizes the equivalence between PDMP and LAMP. See Appendix~\ref{appdx:deferred} for the proof.
\begin{proposition}\label{prop:equivalence}
    Consider the sequences $\{p^t\}$ and $\{\bar p^{t}\}$ from Algorithm~\ref{alg:pdmp} and $\{\nu^t\}$ and $\{\bar{\nu}^{t}\}$ from Algorithm~\ref{alg:lamp}. Assume that $\tau_1\eta\leq 1$ and set $\eta_0=\infty$. Then, for all $t\geq 0$, we have the equivalence
    \begin{equation*}
        p^t = p^{\eta_t}(\nu^t),\quad \bar p^{t+1} =  p^{\eta_{t+1}}(\bar\nu^{t+1}).
    \end{equation*}
\end{proposition}

Therefore, LAMP is a dual implementation of PDMP, tracking the weighted dual averages $\{\nu^t\}$ and $\{\bar\nu^t\}$ and the scalar sequence $\{\eta_t\}$ to implicitly store primal information. By the closed-form solution in~\eqref{def:closed_form_opt}, we only need the primal marginal $\col_r(p^\eta(\nu))$ to compute $\mathcal{M}^{\eta }_\tau(\theta;p^\eta(\nu))$ in~\eqref{step:lamp_mid_theta} and~\eqref{step:lamp_main_theta}, which can be computed using $\mathcal{O}(n+m)$ space as discussed in Subsection~\ref{ssec:log_domain}. Note that Round only involves low-rank corrections, therefore the last step of Algorithm~\ref{alg:lamp} can be performed implicitly with $\mathcal{O}(n+m)$ memory and $\mathcal{O}(nm)$ operations. \textbf{Remark:} Our contribution extends beyond PDMP, as any other mirror prox variant with non-ergodic convergence guarantees can be similarly implemented by tracking the low-dimensional dual variables.

The $\{\eta_t\}$ sequence in~\eqref{step:eta_update} is particularly noteworthy. Observe that $\eta_{t+1}^{-1}$ satisfies the recursion $\eta_{t+1}^{-1}=\tau_1+\eta_{t}^{-1}(1-\tau_1\eta )$, with $\eta_1 =1/\tau_1$ by taking the limit $\eta_0\to\infty$. Solving the recursion gives the generic form
$
    \eta_{t+1}^{-1}={\eta }^{-1}(1-(1-\tau_1\eta )^{t+1}),
$
which becomes $\eta_{t}^{-1}=\tau_1t$ in the limit $\eta \to0$. Therefore, PDMP and LAMP are actually \emph{temperature annealing} methods. Annealing has been highly effective in prior EOT methods~\cite{schmitzerStabilizedSparseScaling2019a,kemertasEfficientAccurateOptimal2025}, and provides a novel perspective on mirror prox methods for OT/EOT.

PDMP has been proposed and analyzed in recent works for entropy-regularized games~\cite{cen2024fast} and OT~\cite{liFastComputationOptimal2025}. In particular,~\cite{liFastComputationOptimal2025} proved that, with certain parameters, Algorithm~\ref{alg:pdmp} achieves $\tO(\|C\|_\infty nm\varepsilon^{-1})$ complexity. Since LAMP is equivalent to PDMP by Proposition \ref{prop:equivalence}, the complexity result recalled below applies to LAMP as well.

\begin{theorem}[{\cite[Theorem 2.2]{liFastComputationOptimal2025}, Informal}]\label{thm:li_informal}
Given $\varepsilon>0$, under the
parameter choices of~\cite{liFastComputationOptimal2025} with $\eta=\mathcal{O}(\varepsilon/(\log m))$, PDMP/LAMP computes
an $\varepsilon$-solution to~\eqref{prob:primal_ot} in
$
\tO(\|C\|_\infty nm\varepsilon^{-1})
$
arithmetic operations.
\end{theorem}

The full statement along with specific parameter choices can be found in Appendix~\ref{appdx:sketch} (see Theorem~\ref{thm:main_convergence_li}). We make two notes about the conditions of Theorem~\ref{thm:main_convergence_li}. First, the variant described in~\cite{liFastComputationOptimal2025} requires entropic regularization for the dual as well. We show in Appendix~\ref{appdx:sketch} that LAMP can similarly implement PDMP with dual regularization. Second, the guarantees of Theorem~\ref{thm:main_convergence_li} hold \textit{only} for a conservative set of parameters. Crucially, the guarantees require $\eta>0$, therefore following the traditional Sinkhorn-type model of solving weakly regularized EOT to solve OT. As observed in~\cite{liFastComputationOptimal2025} and our testing, a more performant and empirically stable set of parameters is $\eta =0$, $\tau_1=\tau_2=1/(2\|C\|_\infty)$, and $\alpha=0.01$. With $\eta=0$, we obtain a \textit{direct} OT method: targeting~\eqref{prob:primal_ot} without the intermediate EOT step. 

The following proposition provides computable optimality guarantees for the sequences generated by Algorithm~\ref{alg:lamp} as well as a sufficient condition for finding approximate saddle-points of~\eqref{prob:spp}. Unlike Theorem~\ref{thm:li_informal}, Proposition~\ref{prop:infeas_bound} provides guarantees for PDMP/LAMP in the fully unregularized regime.
\begin{proposition}\label{prop:infeas_bound}
    Suppose $\tau_1=\tau_2=1/(2\|C\|_\infty)$, $\eta =0$, $\beta=\log 3$, and $\alpha>0$. Let $X^*\in\coup(r,c)$ be a minimizer of~\eqref{prob:primal_ot}. For each iteration of Algorithm~\ref{alg:lamp}, define $X^t:=\diag{r}p^{\eta_t}(\nu^t)$. Then, for all iterations $t\geq 1$ of Algorithm~\ref{alg:lamp} (or, equivalently, Algorithm~\ref{alg:pdmp}), the following statements hold
    \begin{itemize}
        \item[a)] 
        \begin{equation}\label{ineq:main_infeas_bounds}
        \inner{C}{\widetilde X^t-X^*}\leq 4\|C\|_\infty\|\col(X^t)-c\|_1 +\frac{2\|C\|_\infty\log m}{t},
    \end{equation}
    where $\widetilde X^t=\text{Round}(X^t, r, c)$;
    \item[b)] defining
    \begin{equation}\label{def:sum_condition}
\hspace{-1.5cm}    D_t:=\sum_{s=0}^{t-1}\left[\DHa(\bar\theta^{s+1}\|\hat\theta^{s+1})- \DHa(\bar\theta^{s+1}\|\theta^{s})+\KL(\bar X^{s+1}\|X^{s+1})- \KL(\bar X^{s+1}\|X^{s})\right],\end{equation}
    and $\bar X^t:=\diag{r}p^{\eta_t}(\bar\nu^t)$, then
    \begin{equation}\label{ineq:spp_midpoint_conv}
        \inner{C}{\text{Round}(\hat X^t,r,c)-X^*}\leq\frac{2\|C\|_\infty[\log m+(1+\alpha)\log 2+D_t]}{t},
    \end{equation}
    where $\hat X^t$ is the ergodic iterate
    \begin{equation}
        \hat X^t=\frac{1}{t}\sum_{i=1}^t \bar X^i.
    \end{equation}
    
    \end{itemize}
\end{proposition}

The right-hand side of~\eqref{ineq:main_infeas_bounds} decomposes the rounded objective error into two terms: the column infeasibility and an $\mathcal{O}(1/t)$ initialization bias. Therefore, the column infeasibility and iteration count provide a \emph{computable, last-iterate optimality certificate for LAMP}. As shown in our numerical results, LAMP finds feasible points quite rapidly. The remaining open theoretical question is then to prove the convergence of $X^t$ to a feasible point. The feasibility convergence argument of~\cite{liFastComputationOptimal2025} relied on a trajectory-wise analysis of the dual variables, resulting in several of the conservative parameter choices to control error terms. Extending a similar trajectory-wise argument to the $\eta=0$ case with performant parameters seems to require a substantially more delicate analysis, which we leave to future work.

The second part of Proposition~\ref{prop:infeas_bound} gives a complementary midpoint guarantee. If the partial sums $D_t$ are $o(t)$, then ergodic primal convergence follows. Indeed, additional numerical experiments in Appendix~\ref{appdx:analysis} show that~\eqref{def:sum_condition} is benign across a variety of OT problems, with $D_t$ rapidly becoming nonpositive and appearing to converge to zero. If $D_t$ can be uniformly bounded from above, then LAMP maintains the classical $\mathcal{O}(1/t)$ Mirror Prox-type ergodic convergence guarantee~\cite{nemirovskiProxMethodRateConvergence2006}. As stated in Section~\ref{sec:intro}, explicitly tracking an $\mathcal{O}(nm)$ primal iterate is impractical for large-scale applications, hence non-ergodic guarantees remain an important open problem.

\section{Implementation Details}\label{sec:implementation}

\subsection{Log-Domain Operation}\label{ssec:log_domain}
 We compute the column marginal $\col_r(p^{\eta_t}(\nu))$ using a two-step process. First, we compute the row-wise log-normalization constants $\log Z_i$ using the LogSumExp trick: $\LSE(x)=\LSE(x-\max_ix_i)+\max_ix_i$. Denoting $m_i=\max_j\{-\eta_t^{-1}(C_{ij}+2\|C\|_\infty\nu_{j})\}$, we have
\begin{align*}
    \log Z_i &= \LSE(-\eta_t^{-1}(C_{i:}+2\|C\|_\infty\nu)-m_i) + m_i,\\
    \col_r(p^{\eta_t}(\nu))_j&=\sum_{i=1}^n r_i\exp\left[-\frac1{\eta_t}(C_{ij}+2\|C\|_\infty\nu_{j})-\log Z_i\right].
\end{align*}
Each term $-\eta_t^{-1}(C_{ij}+2\|C\|_\infty\nu_{j})$ is computed on-the-fly and accumulated into either an $\mathcal{O}(n)$ buffer (for $\log Z_i$ terms) or an $\mathcal{O}(m)$ buffer (for $\col_r(p^{\eta_t}(\nu))_j$ terms), leading to the claimed $\mathcal{O}(n+m)$ space complexity of LAMP.

\subsection{Optimized Reductions}\label{subsec:reduction}
Log-domain computations are commonly used to stabilize EOT algorithms~\cite{peyre2019computational,schmitzerStabilizedSparseScaling2019a}, and are common in widely-distributed OT/EOT packages~\cite{flamaryPOTPythonOptimal2021b}. Our primary challenge was to perform the operations in a computationally efficient manner. Log-domain computation requires three reduction operations: row-wise $\max$, row-wise $\LSE$, and finally a column-wise sum.

To reduce reduction overhead, we utilize custom kernels leveraging warp-tiling and kernel/loop fusion. Warp-tiling uses a single warp to perform each entry of the reduced vector, with threads within a warp coalescing global memory accesses and using warp-level reductions. Fusion performs the $\max$ and the LSE reductions in the same loop, reducing the number of memory accesses at the cost of additional $\exp$ and branching operations. Further details can be found in Appendix~\ref{appdx:numerical}.

After applying warp-tiling (``WT'') and loop fusion (``Fused''), our measured reduction kernel latency outperforms the optimized PyKeOps library~\cite{charlier2021kernel}, as shown in Table~\ref{tab:kernel_comparison}. At $n=1024$, the naive kernel does not achieve full GPU occupancy (32768 threads), hence warp-tiling improves performance by 38$\times$. However, as the problem size increases, the naive kernel comes closer to achieving full SM occupancy, ultimately removing concurrency benefits from warp-tiling. Even when the GPU is at full occupancy, warp-tiling results in improved memory access patterns and a $1.4\times$ improvement over the naive kernel. The ``fused'' kernel further improves memory access overhead as $n$ increases, gaining an additional $1.25\times$ speedup. 
\begin{table}[t]
    \caption{Ablation study of kernel optimizations with comparison to the LogSumExp reduction from the PyKeOps library on an RTX 4090 GPU. Across problem sizes, the optimized LogSumExp kernels outperform the widely-used PyKeOps reduction kernel baseline.}
    \label{tab:kernel_comparison}
    \centering
\begin{tabular}{lcccc}
\toprule
&\multicolumn{4}{c}{\textbf{Wall-clock time (ms)}}\\
 \textbf{Kernel}  & $n=1024$ & $n=4096$ & $n=16384$ & $n=65536$ \\
\midrule
Baseline &  $3.77\pm 0.092$ & $14.8\pm 0.24$ & $58.9\pm 0.86$ & $474.6\pm 7.04$ \\
WT &  $\bm{0.1\pm0.008}$ & $1.37\pm0.044$ & $21.3\pm0.34$ & $337.7\pm5.36$ \\
Fused+WT &  $0.12\pm 0.0078$ & $\bm{1.30\pm 0.0426}$ & $\bm{17.13\pm 0.26}$ & $\bm{252.5\pm 2.90}$ \\
PyKeOps &  $1.79\pm0.016$ & $6.96\pm 0.015$ & $27.8\pm 0.026$ & $332\pm 0.16$ \\
\bottomrule
\end{tabular}
\end{table}

\section{Numerical Experiments}\label{sec:numerical}

In this section we compare Algorithm~\ref{alg:lamp} with alternative first-order OT/EOT algorithms. Further details, including problem generation and hyperparameter choices, can be found in Appendix~\ref{appdx:numerical}. All experimental/solver code is implemented in Julia and is publicly available.\footnote{
\href{https://github.com/mxburns2022/CuLAMP.jl}{https://github.com/mxburns2022/CuLAMP.jl}}\\
\noindent\textbf{Competitor Solvers}: We compare to dense-reduction first-order baselines, matching the computational model targeted by LAMP. Baseline solvers include Sinkhorn algorithm~\cite{altschulerNearlinearTimeApproximation2017}, accelerated primal-dual adaptive mirror descent (APDAMD)~\cite{linEfficiencyEntropicRegularized2022}, accelerated Sinkhorn~\cite{linEfficiencyEntropicRegularized2022}, the Bregman hybrid primal-dual (HPD) algorithm~\cite{chambolleAcceleratedBregmanPrimalDual2022}, annealed Sinkhorn using warm starts as described in~\cite{kemertasEfficientAccurateOptimal2025}, and Dual Extrapolation as described in Algorithm 3 of~\cite{jambulapatiDirectTildelbraceOrbrace2019a}. Other methods were tested, such as APDAGD~\cite{dvurechenskyComputationalOptimalTransport2018} and Greenkhorn~\cite{altschulerNearlinearTimeApproximation2017}, however we focused on the top performing solvers for the purposes of this work.\footnote{APDAGD showed significant instabilities in the small-$\eta$ regime as observed in~\cite{linEfficiencyEntropicRegularized2022}, while Greenkhorn is unable to efficiently utilize parallel hardware due to its update scheme.} LAMP and PDMP are equivalent up to numerical error for each iteration, hence we do not compare them directly. Solvers are set to terminate upon reaching a primal-dual gap $\leq 10^{-10}$, where the dual functions for each formulation are given in Appendix~\ref{appdx:dual}. \\

\noindent\textbf{Benchmarks}:  For head-to-head testing, we use instances from the DOTmark set~\cite{schrieberDOTmarkBenchmarkDiscrete2017} with $\|\cdot\|_p^p$ ground costs (except $p=\infty$, where ground costs are $\|\cdot\|_\infty$). We additionally draw benchmark instances from the cell data similarity problems from~\cite{huizingOptimalTransportImproves2022}. \\

Fig.~\ref{fig:eot_broad} compares fixed-$\eta$ EOT solver performance in $n=m=1024$ DOTmark problems with $\ell_2^2$ ground costs. The y-axis on each plot shows the EOT primal gap plus marginal infeasibility, hence measuring convergence in the entropy-regularized problem alone. For $\eta\geq 10^{-4}$, Sinkhorn and Accelerated Sinkhorn significantly outperform the other methods tested, though LAMP outperforms several primal-dual methods (HPD and APDAMD) in this regime. However, LAMP's advantage is clear for the weakly regularized case ($\eta=10^{-6}$), where LAMP converges extremely rapidly while the other solvers appear to converge sublinearly.

\begin{figure}[t]
    \centering
    \begin{subfigure}[b]{0.35\linewidth}
        \includegraphics[width=\linewidth]{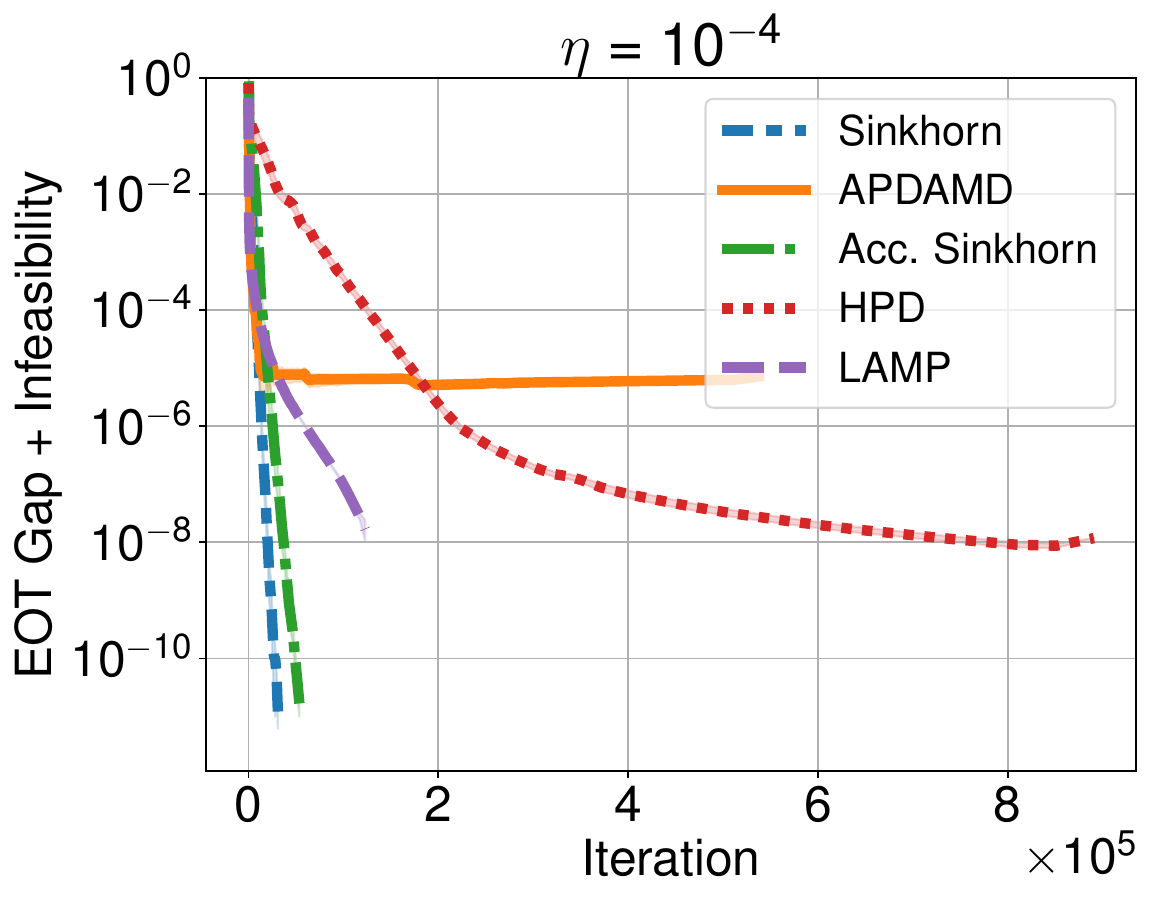}
    \end{subfigure}
    \begin{subfigure}[b]{0.35\linewidth}
        \includegraphics[width=\linewidth]{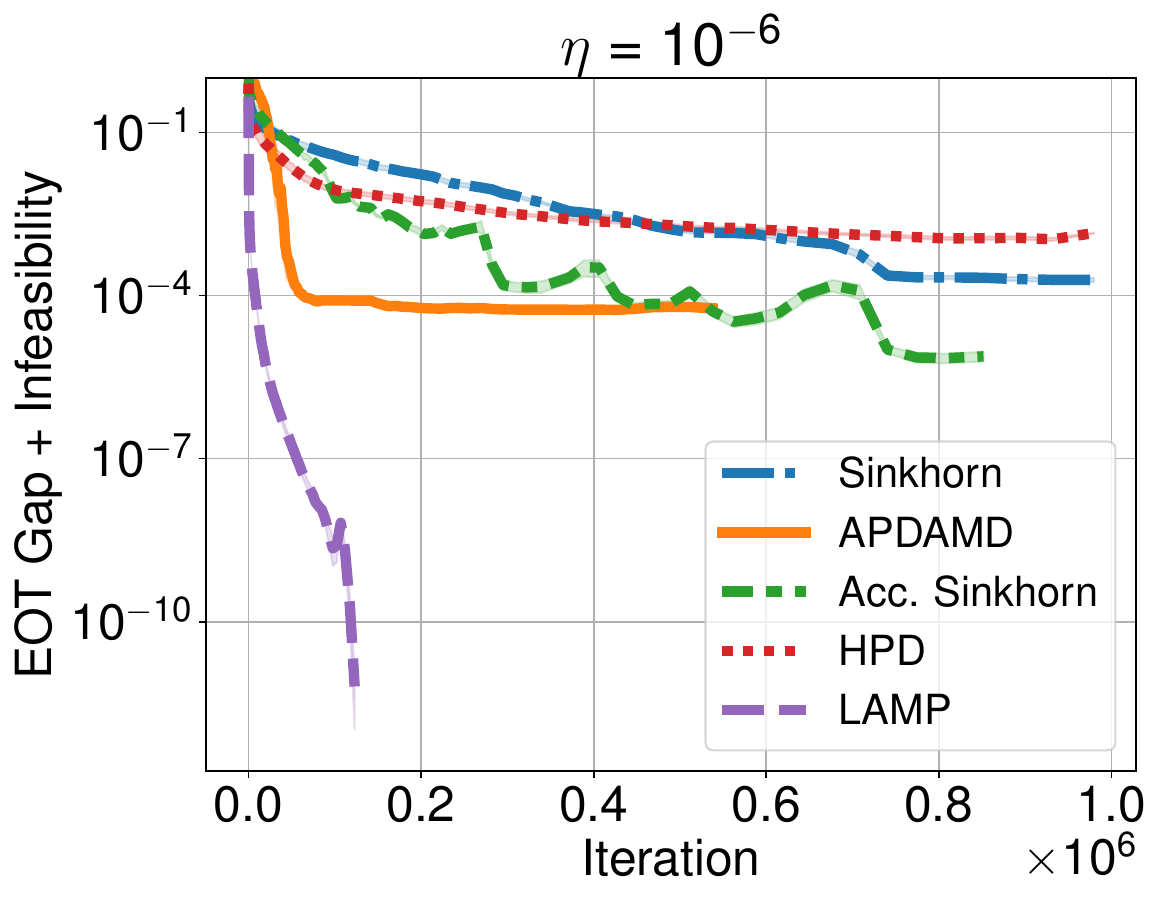}
    \end{subfigure}
    \caption{Comparison of  various OT solvers on $n=1024$ DOTmark problems with Euclidean ground costs. LAMP is particularly effective in the weakly-regularized $\eta=10^{-6}$ regime. }
    \label{fig:eot_broad}
\end{figure}

We now focus on comparing unregularized  LAMP ($\eta=0$) to Sinkhorn-based methods ($\eta>0$) as the primary $\mathcal{O}(n+m)$ space baseline. Furthermore, we add temperature-annealed Sinkhorn with warm starts as described in~\cite{kemertasEfficientAccurateOptimal2025} as an annealing baseline, with $\eta_t=\max\{\eta_iq^t,\eta_f\}$ for $q\in (0,1)$. All solvers use on-the-fly, kernel-based distances with the warp-tiling+fusion optimizations discussed in Subsection~\ref{subsec:reduction}. For our comparisons, we utilize the combined objective gap + infeasibility 
\begin{equation}\label{eq:combined}
    \inner{C}{X^t-X^*}+\|C\|_\infty(\|\col(X^t)-c\|_1 + \|\row(X^t)-r\|_1),
\end{equation}
which acts as an upper bound on the cost of the rounded iterate $\operatorname{Round}(X^t,r,c)$  when $X^t$ is either row or column feasible (true for both Sinkhorn and LAMP). Exact OT costs are computed using \texttt{emd2} from PythonOT~\cite{flamaryPOTPythonOptimal2021b}.

As illustrated by Proposition~\ref{prop:infeas_bound}(a), LAMP convergence speed is dependent on $\|C\|_\infty$, which in turn depends on the underlying metric. Fig.~\ref{fig:metric_comp} compares each kernel OT solver on $n=m=4096$ DOTmark problems with varying ground cost. For $\ell_\infty$ and $\ell_1$ ground costs, LAMP converges extremely rapidly, outperforming both the Sinkhorn baselines. Since $\|C\|_\infty$ scales $\mathcal{O}(\sqrt{n})$ for both costs, the initial bias term in~\eqref{ineq:main_infeas_bounds} is relatively benign. In contrast, LAMP exhibits slower convergence for $\ell_2^2$, where $\|C\|_\infty=\mathcal{O}(n)$. LAMP is competitive for much of the $\ell_2^2$ trajectory, though Sinkhorn eventually overtakes it with well-chosen regularization. This, however, highlights another benefit of LAMP. In contrast to Sinkhorn, where the value of $\eta$ may require problem-dependent tuning, unregularized LAMP ($\eta=0$) required no problem-specific tuning in our tests to obtain competitive performance.

LAMP further outperforms Sinkhorn in data similarity computation, shown in Fig.~\ref{fig:cell_similarity}, where the dataset consists of cell omics data from~\cite{liu2019deconvolution} and preprocessed by~\cite{huizingOptimalTransportImproves2022}. We define the costs $C_{ij}$ using four different similarity kernels: $\ell_1$ and $\ell_2^2$ costs, cosine similarities, and Pearson correlations. In each case, $\|C\|_\infty\approx 1$, leading to LAMP converging significantly faster than the Sinkhorn solvers. Appendix~\ref{appdx:numerical} provides additional description on the problem setup and specific runtime breakdowns.

\begin{figure}
    \centering
    \begin{subfigure}[]{0.3\linewidth}
        \includegraphics[width=\linewidth]{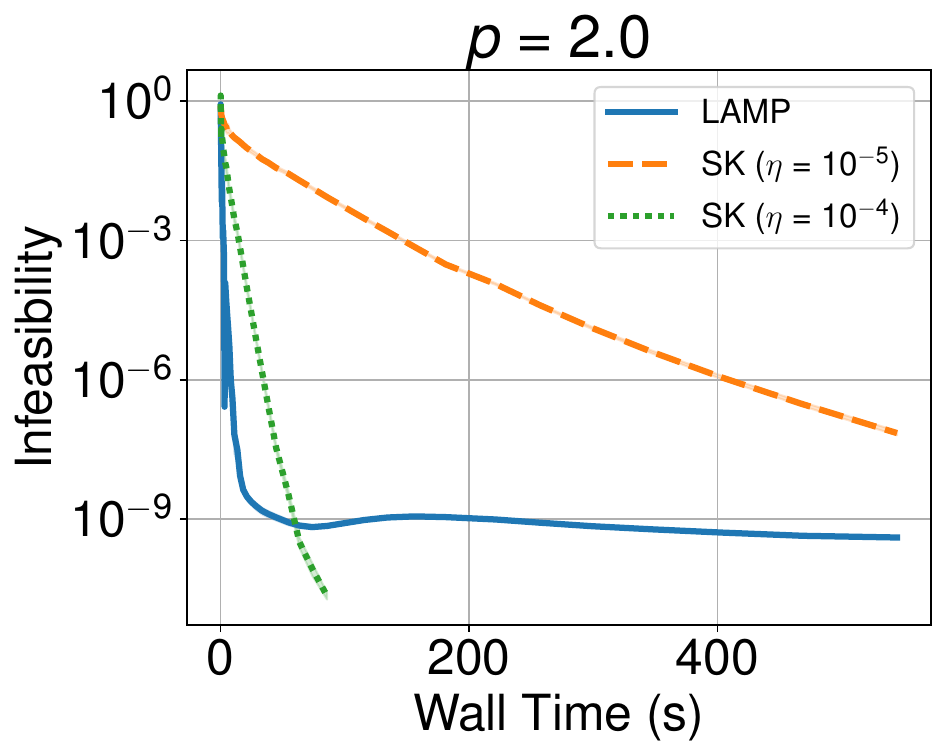}
    \end{subfigure}
    \begin{subfigure}[]{0.3\linewidth}
        \includegraphics[width=\linewidth]{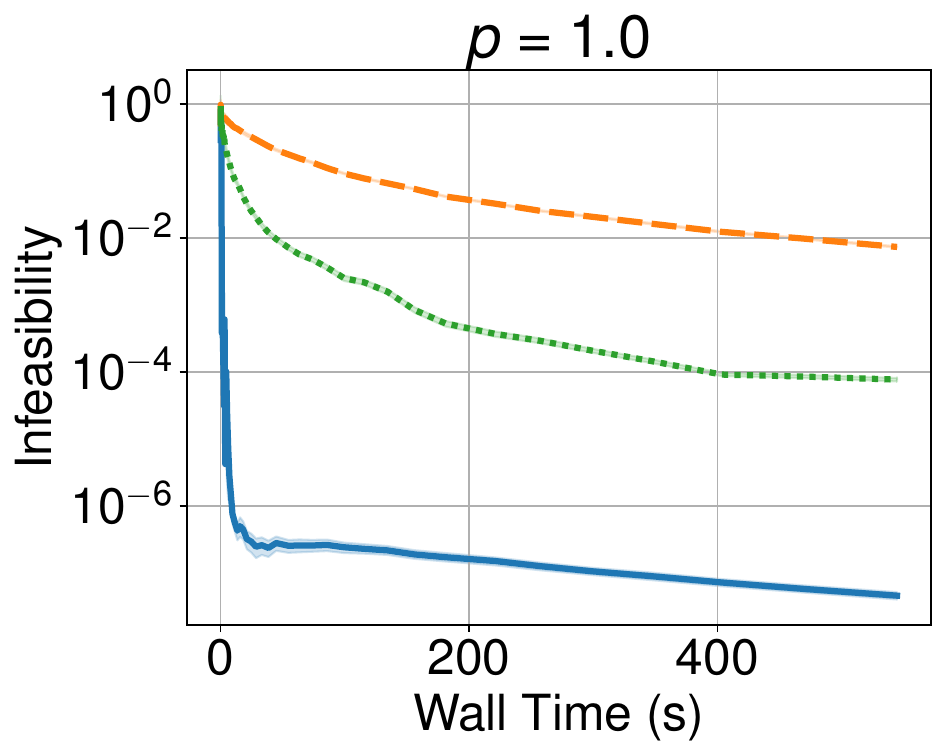}
    \end{subfigure}
    \begin{subfigure}[]{0.3\linewidth}
        \includegraphics[width=\linewidth]{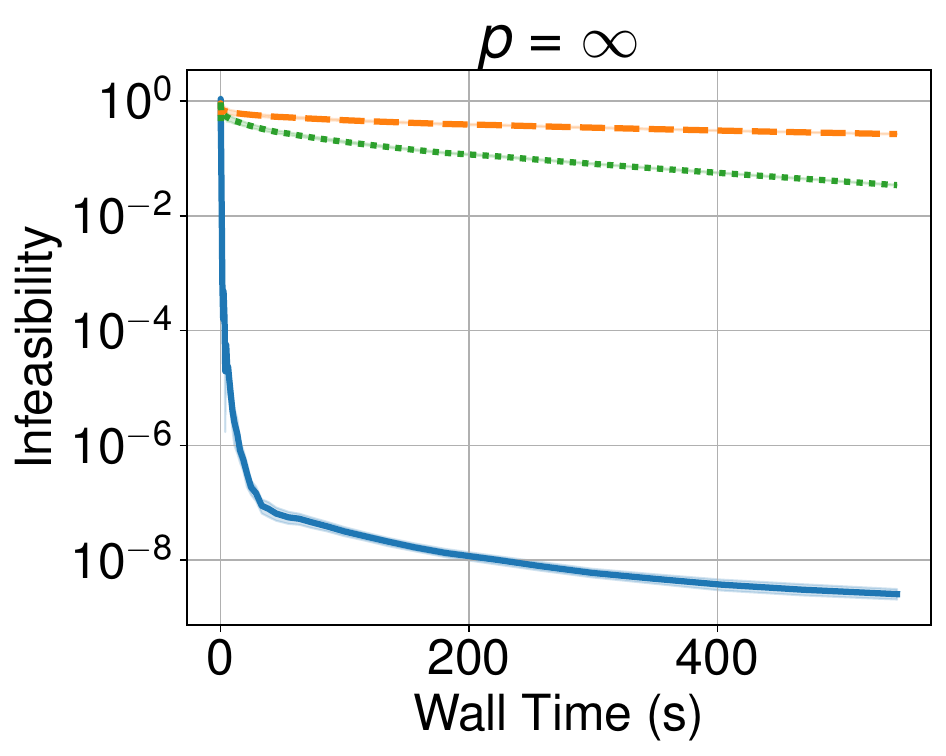}
    \end{subfigure}
    \begin{subfigure}[]{0.3\linewidth}
        \includegraphics[width=\linewidth]{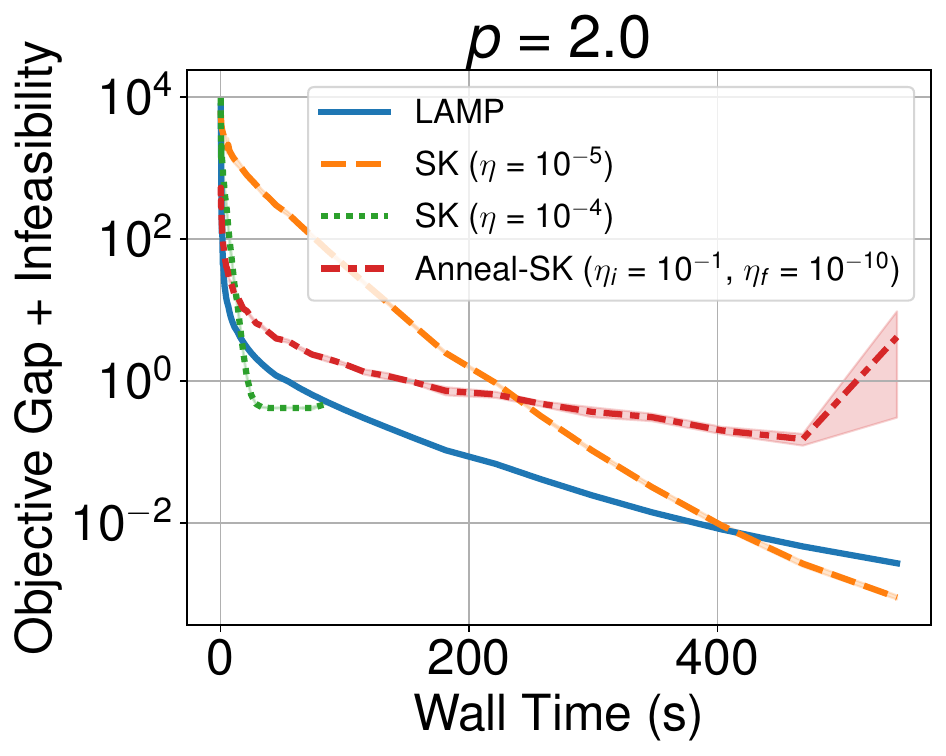}
    \end{subfigure}
    \begin{subfigure}[]{0.3\linewidth}
        \includegraphics[width=\linewidth]{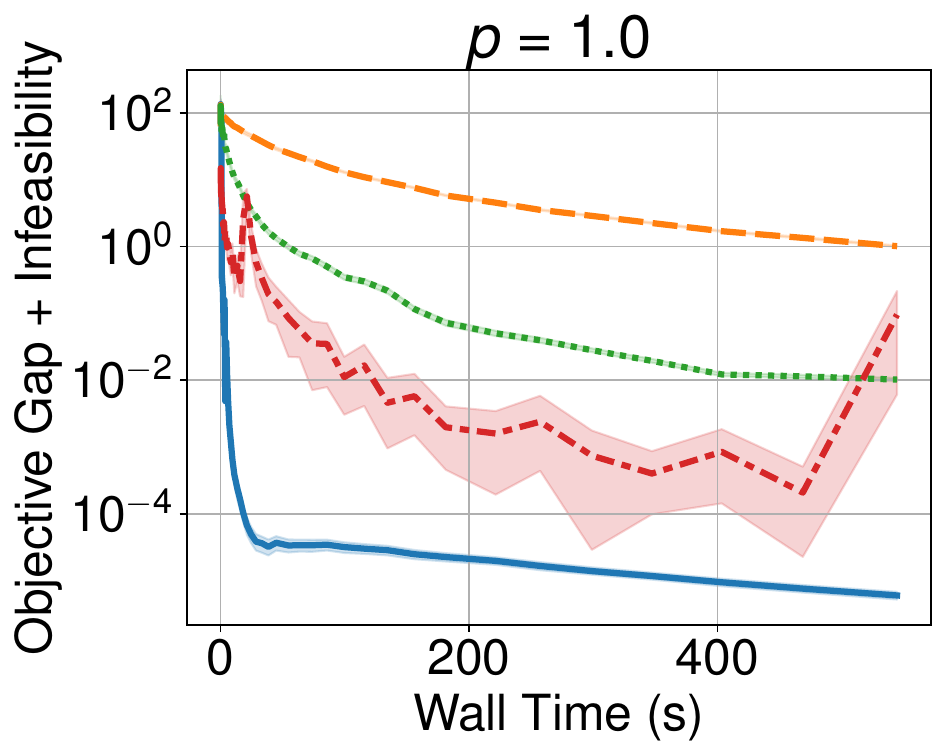}
    \end{subfigure}
    \begin{subfigure}[]{0.3\linewidth}
        \includegraphics[width=\linewidth]{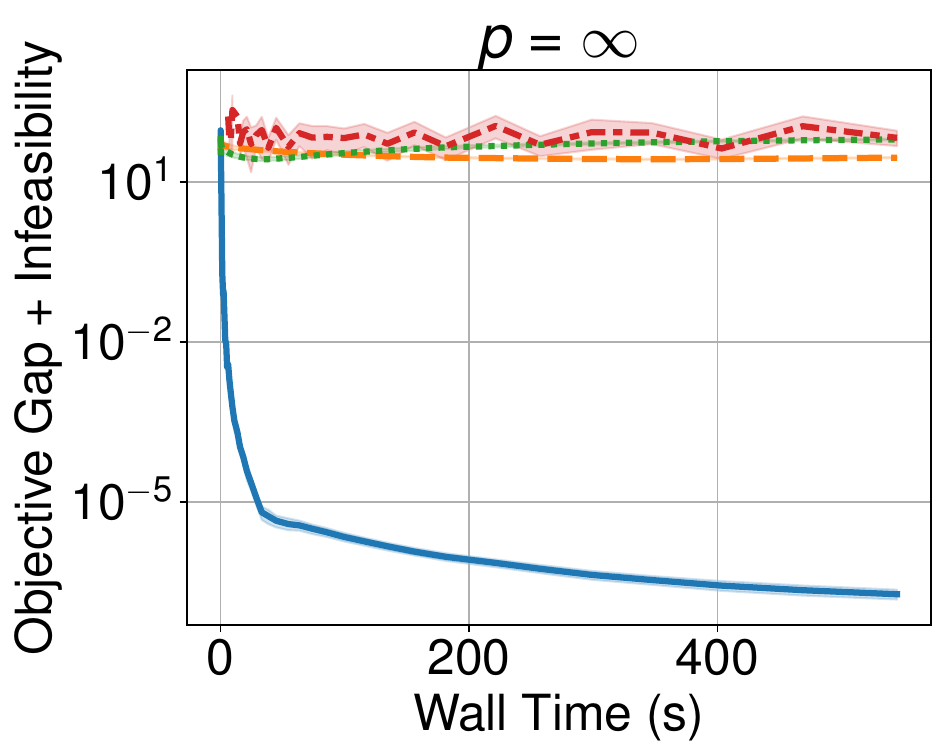}
    \end{subfigure}
    \caption{Solution trajectories on $n=m=4096$ DOTmark instances with varying ground metrics comparing [top] infeasibility and [bottom] objective gap + infeasibility as described in~\eqref{eq:combined}. LAMP's performance is metric dependent, with particularly high performance for costs with $\|C\|_\infty=o(n)$. We omit Anneal-SK (using $(\eta_i,\eta_f,q)=(10^{-1},10^{-10},0.8)$) from the infeasibility plot, as at each outer iteration the iterate is feasible (except for the last timed-out iteration).}
    \label{fig:metric_comp}
\end{figure}

Summarizing the numerical comparisons in this section, we offer the following broad guidance. LAMP is highly performant and well-suited to problems where the cost scale $\|C\|_\infty$ is moderate or naturally bounded and high accuracy is desired. If a low-accuracy solution will suffice, or $\|C\|_\infty$ is particularly large/dimension-dependent, then tuned Sinkhorn will likely be more performant. Note that the behavior of $\|C\|_\infty$ depends on both the problem data and the choice of metric. On the grid-based DOTmark instances shown in Fig.~\ref{fig:metric_comp}, the scaling of $\|C\|_\infty$ depended on the choice of metric ($\|C\|_\infty=\mathcal{O}(n)$ for $p=2$ and $\|C\|_\infty=\mathcal{O}(\sqrt{n})$ for $p=1$ or $p=\infty$). In contrast, $\|C\|_\infty=\mathcal{O}(1)$ for the cell similarity problems shown in Fig.~\ref{fig:cell_similarity}, even for squared Euclidean costs. Since $\|C\|_\infty$ is easily computable (and is required to apply LAMP regardless), the decision between algorithms can therefore be made on an instance-by-instance basis.

Finally, we provide a proof-of-concept demonstration of LAMP on large-scale OT problems. Using the kernel-based LAMP code, we computed 512 x 512 color transfer maps with a 4 hour time limit and $\ell_2^2$ ground costs, shown in Fig.~\ref{fig:ctransfer}. Explicit primal-dual methods (such as PDMP or accelerated mirror descent methods~\cite{dvurechenskyComputationalOptimalTransport2018,linEfficiencyEntropicRegularized2022}) would require $\approx512$ GB (the equivalent of eleven L40s GPUs), while LAMP is able to run on a single L40s GPU using approximately 38 MB.

\begin{figure}[t]
    \centering
    \includegraphics[width=\linewidth]{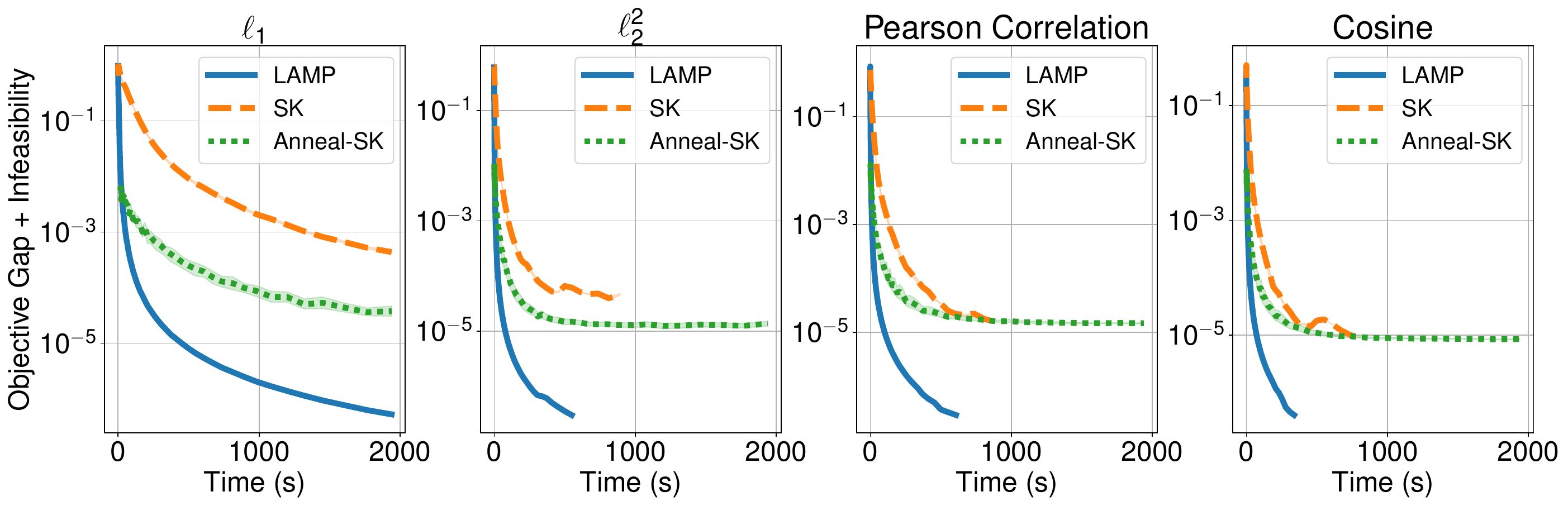}
    \caption{Comparison between Sinkhorn with $\eta=10^{-4}$ (SK), Annealed Sinkhorn (Anneal-SK) with $(\eta_i,\eta_f,q)=(10^{-2},10^{-4},0.95)$, and LAMP ($\eta=0$) in cell similarity tasks using the cell omics datasets from~\cite{huizingOptimalTransportImproves2022}. Costs are computed on the fly by comparing cell features using various similarity metrics. For these plots, we use 20 cells and 5000 features ($n=m=5000$) and average over 10 problem instances. In each case, LAMP significantly outperforms the Sinkhorn-based methods.}
    \label{fig:cell_similarity}
\end{figure}
\section{Discussion}\label{sec:discussion}

An independent line of research from~\cite{assadiSemiStreamingBipartiteMatching2022} proved a result similar to Proposition~\ref{prop:equivalence}, achieving  $\mathcal{O}(n+m)$ space complexity using an auxiliary dual sequence similar to $\nu^t$ as well as a scalar $\lambda_t$. One of the main subroutines in~\cite{assadiSemiStreamingBipartiteMatching2022} is dual extrapolation~\cite{jambulapatiDirectTildelbraceOrbrace2019a}, which includes the update
\begin{equation}\label{eq:dextrap_update}
X^{t+1}=\argmin_{X\in\Delta^{n\times m}}\{\inner{v^t}{X}+\tau\KL(X\|X^t)\},
\end{equation}
where $v^t\in \R^{n\times m}$ can be computed in $\mathcal{O}(nm)$ time using additional $\mathcal{O}(n+m)$ dual variables.
    
\begin{figure}[t]
\centering

\begin{minipage}[c]{0.5\textwidth}
    \centering

    \begin{minipage}[t]{0.45\linewidth}
        \centering
        \includegraphics[width=\linewidth]{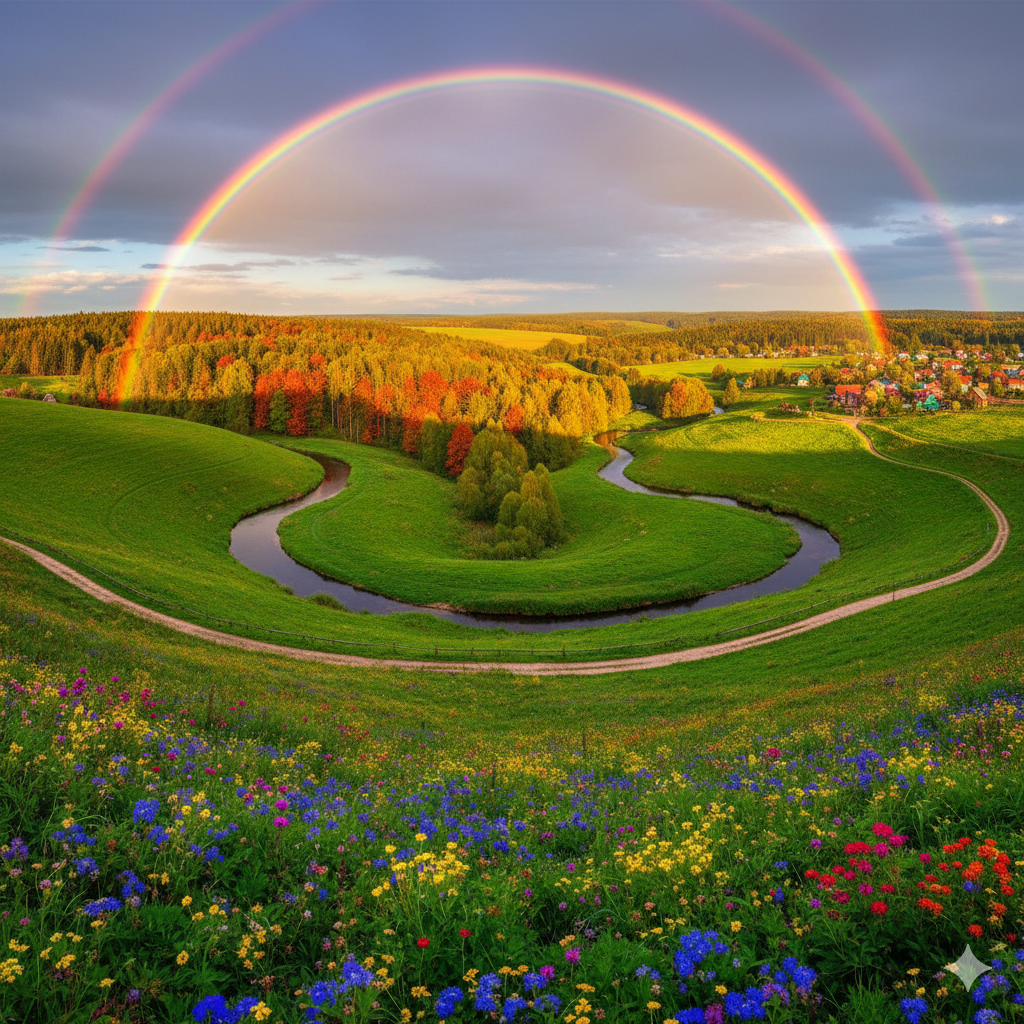}
    \end{minipage}
    \begin{minipage}[t]{0.45\linewidth}
        \centering
        \includegraphics[width=\linewidth]{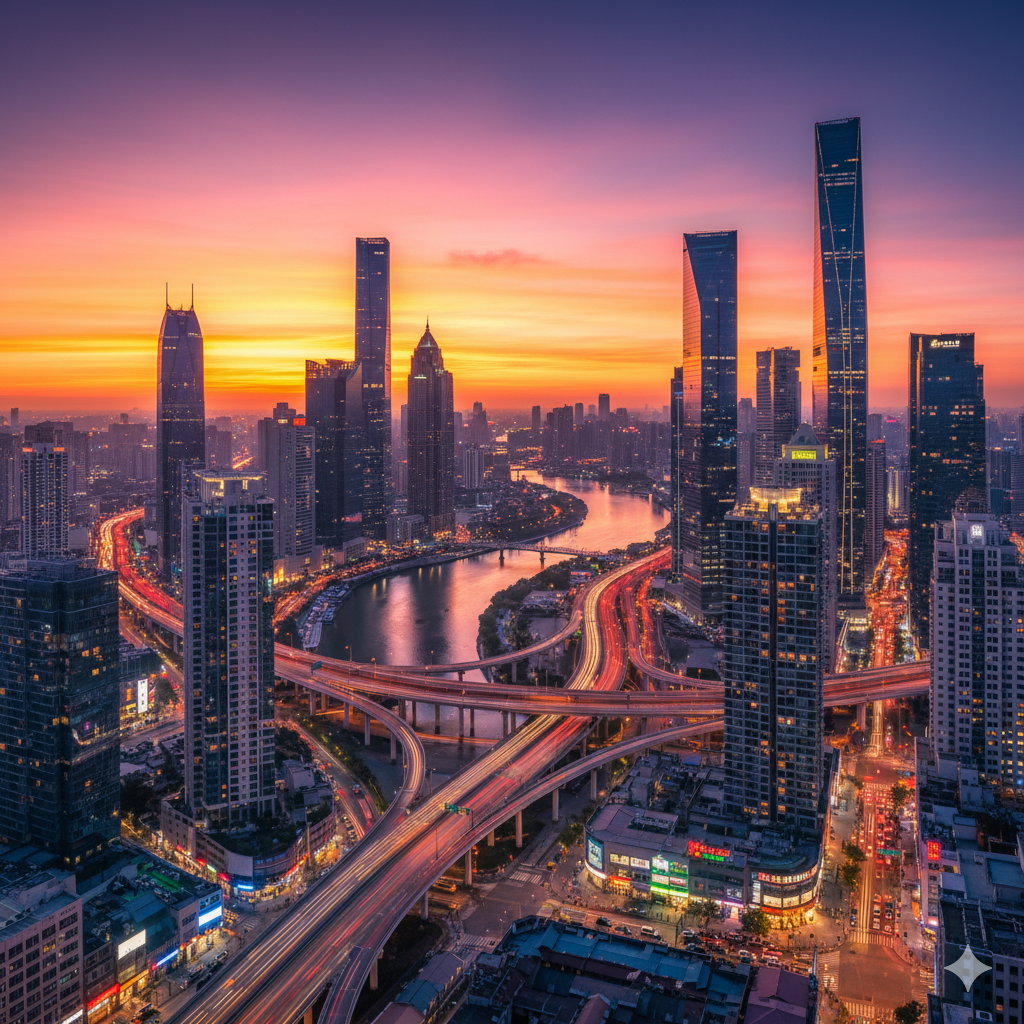}
    \end{minipage}

    \vspace{0.03\textwidth}

    \begin{minipage}[t]{0.45\linewidth}
        \centering
        \includegraphics[width=\linewidth]{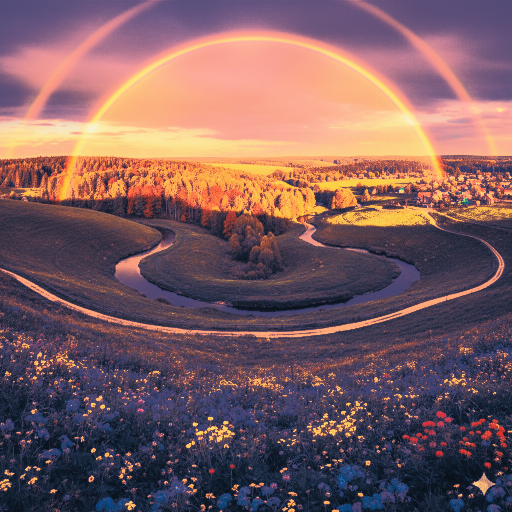}
    \end{minipage}
    \begin{minipage}[t]{0.45\linewidth}
        \centering
        \includegraphics[width=\linewidth]{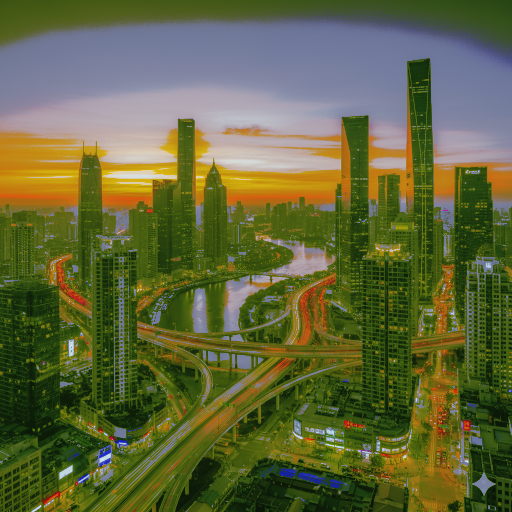}
    \end{minipage}
\caption{Color transfer on 512 x 512 images with three color channels. The original images [top] were generated by Google Gemini. The color transfers [bottom] were computed with a 4-hour timeout running on an NVIDIA L40s.}
    \label{fig:ctransfer}
\end{minipage}
\hfill
\begin{minipage}[c]{0.45\textwidth}
    \centering
    \includegraphics[width=\linewidth]{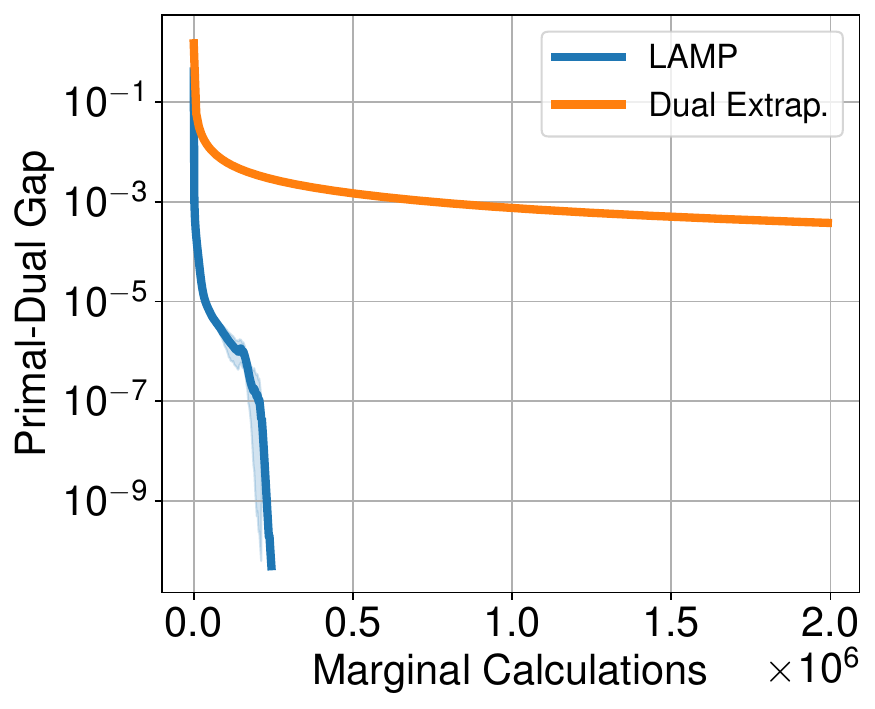}
    \caption{Comparison between the dual extrapolation method proposed in~\cite[Algorithm 3]{jambulapatiDirectTildelbraceOrbrace2019a} and LAMP. The x-axis counts the number of marginal computations performed by each algorithm, which each requires $\mathcal{O}(nm)$ time.}
    \label{fig:dextrap}
\end{minipage}
\end{figure}
The multiplicative form of~\eqref{def:closed_form_opt} is the key mechanism used in the proof of Lemma~\ref{lem:primal_md_eq_dual_avg}. Since~\eqref{eq:dextrap_update} also leads to a multiplicative update, both LAMP and dual extrapolation are able to implicitly store primal information using extra dual sequences. However, the dual extrapolation method underlying the proposed algorithm in~\cite{assadiSemiStreamingBipartiteMatching2022} has two shortcomings. First, its convergence guarantees require an average primal iterate. To maintain convergence guarantees with only dual iterates,~\cite{assadiSemiStreamingBipartiteMatching2022} propose a recovery procedure dependent on pointer-based data structures, making the subroutine difficult to efficiently map to GPUs. Second, the dual extrapolation method from~\cite{jambulapatiDirectTildelbraceOrbrace2019a} is generally slow in practice. Fig.~\ref{fig:dextrap} compares dual extrapolation to LAMP on $n=1024$ DOTmark problems with $\ell_2^2$ costs, showing that LAMP significantly outperforms the competing method. Since the dual extrapolation subroutine dominates the runtime complexity in~\cite{assadiSemiStreamingBipartiteMatching2022} (see Theorem 4.1 therein), we reasonably conjecture that the conclusions of Fig.~\ref{fig:dextrap} hold for the linear-space OT framework of~\cite{assadiSemiStreamingBipartiteMatching2022} (see Theorem 5.2 of~\cite{assadiSemiStreamingBipartiteMatching2022}). 

\section*{Limitations and Future Work}
\textbf{Theoretical Gaps} As noted in the main text, proving non-ergodic optimality guarantees for $\eta =0$ and $\tau_1,\tau_2=1/(2\|C\|_\infty)$ is a necessary next step to justify the ``empirically performant'' parameter choices. Proposition~\ref{prop:infeas_bound} provides practically useful bounds for the performant regime, however the lack of a full convergence proof is our primary theoretical limitation. 

\textbf{Extension to Second-Order Methods} Recent second-order OT solvers utilizing Krylov methods~\cite{kemertasTruncatedNewtonMethod2024,kemertasEfficientAccurateOptimal2025,wu2025pins} and/or sparse Newton iterations~\cite{wu2025pins,pan2026inexact} have shown significant speedups over first-order baselines. These methods utilize the traditional dual of~\eqref{prob:primal_eot} and are often built on a temperature-annealing framework~\cite{kemertasEfficientAccurateOptimal2025}. In this work, we have shown that mirror prox methods built on the alternative problem~\eqref{prob:spp} (the dual can be found in Appendix~\ref{appdx:dual}) can substantially outperform comparable first-order methods built on the traditional (E)OT formulation.  Our claim is limited: LAMP is simple to implement, linear space, admits a computable last-iterate certificate (Proposition~\ref{prop:infeas_bound}(a)), and is highly competitive in its algorithmic class (highly parallel first-order methods), not that it is a universally dominant algorithm. Our results therefore suggest the development of second-order methods targeting~\eqref{prob:spp} as a promising direction for our future work.


\bibliographystyle{abbrv}
\bibliography{refs}
\appendix
\onecolumn
\section{Literature Review}\label{app:review}
Here we provide a brief review of relevant first-order methods for OT/EOT. Since our work focuses on first-order methods leveraging parallel operations, we omit discussion of combinatorial OT algorithms, such as~\cite{lahn2019graph}. For simplicity, we use $\mathcal{O}(n+m)$ to indicate all linear complexity bounds (e.g., $\mathcal{O}(n)$, $\mathcal{O}(m)$, $\mathcal{O}(n+m)$,  and $\mathcal{O}(\max\{n,m\})$).
\begin{table}[h]
    \caption{\footnotesize Non-exhaustive summary of related works on first-order methods for OT. ``Complexity'' reflects the computational complexity required to find an $\varepsilon$-solution to the primal OT problem. For simplicity, we assume $\|C\|_\infty\geq 1$, and we state linear dependence on $m$ or $n$ as $m+n$. Alternative references are provided for the complexity bounds if they differ from the results of the original work. ``Space complexity'' reflects the complexity required to store the algorithm iterates, excluding the cost to store the cost matrix $C$, which can often be computed on-the-fly for kernel-based costs. {For DROT, $R_{DR}:=\|Y_0-Y_*\|$ is the initial distance from a fixed point of the Douglas-Rachford operator~\cite[Equation 6]{maiFastAccurateSplitting2021a}}. For PDOT, $\delta$ is the data precision ($\Delta$ in the author's original notation, see Definition 2 in~\cite{luPDOTPracticalPrimalDual2024}).
    Acronyms are as follows: ``APDA(G,M)D'' are Adaptive Primal-Dual Accelerated (Gradient, Mirror) Descent, AAM is Accelerated Alternating Minimization, DROT is Douglas-Rachford OT, HPD is Hybrid Primal-Dual, PDASMD is Primal-Dual Accelerated Stochastic Mirror Descent, and PDOT is Primal-Dual OT, based on the primal-dual hybrid gradient method of~\cite{luCuPDLPjlGPUImplementation2024}. } 
    \label{tab:summary}
    \centering
    \begin{tabular}{|l|l|l|}\hline
        Algorithm(s) &  Complexity & Space Complexity\footnotemark \\\hline\hline
         Sinkhorn~\cite{cuturiSinkhornDistancesLightspeed2013} & 
         $\tO(\|C\|_\infty ^2nm\varepsilon^{-2})$~\cite{dvurechenskyComputationalOptimalTransport2018} & $\bm{\mathcal{O}(n+m)}$ \\ 
        Greenkhorn~\cite{altschulerNearlinearTimeApproximation2017} & 
         $\tO(\|C\|_\infty^2 nm\varepsilon^{-2})$~\cite{linEfficiencyEntropicRegularized2022} & $\bm{\mathcal{O}(n+m)}$ \\ 
          APDAGD~\cite{dvurechenskyComputationalOptimalTransport2018} & $\tO(\|C\|_\infty nm(n+m)^{1/2}\varepsilon^{-1})$~\cite{linEfficiencyEntropicRegularized2022} & $\mathcal{O}(nm)$ \\
         Dual Extrapolation~\cite{jambulapatiDirectTildelbraceOrbrace2019a}
           & $\bm{\tO(\|C\|_\infty nm\varepsilon^{-1})}$& $\bm{\mathcal{O}(n+m)}$\cite{assadiSemiStreamingBipartiteMatching2022} \\

         AAM~\cite{guminovCombinationAlternatingMinimization2021} & $\tO(\|C\|_\infty nm(n+m)^{1/2}\varepsilon^{-1})$  & $\mathcal{O}(nm)$ \\
          DROT~\cite{maiFastAccurateSplitting2021a} & $\bm{\O(R_{DR}nm\varepsilon^{-1})}$ & $\mathcal{O}(nm)$\\
         APDAMD~\cite{linEfficiencyEntropicRegularized2022} & $\tO(\|C\|_\infty nm(n+m)^{1/2}\varepsilon^{-1})$ & $\mathcal{O}(nm)$\\
         Acc. Sinkhorn~\cite{linEfficiencyEntropicRegularized2022}  & $\tO(\|C\|_\infty^{4/3} nm(n+m)^{1/3}\varepsilon^{-4/3})$ & $\bm{\mathcal{O}(n+m)}$\\
         HPD~\cite{chambolleAcceleratedBregmanPrimalDual2022} & $\tO(\|C\|_\infty nm(n+m)^{1/2}\varepsilon^{-1})$ & $\mathcal{O}(nm)$\\
         PDASMD~\cite{luoImprovedRateFirst2023} & $\bm{\tO(\|C\|_\infty nm\varepsilon^{-1})}$ & $\mathcal{O}(nm)$\\
         PDOT~\cite{luPDOTPracticalPrimalDual2024} & $\tO(nm(\|C\|_\infty\delta^{-1}(n+m)^{3}+(n+m)^{3/2}))$\footnotemark & $\mathcal{O}(nm)$
         \\
         PDMP~\cite{liFastComputationOptimal2025} & $\bm{\tO( \|C\|_\infty nm\varepsilon^{-1})}$ & $\mathcal{O}(nm)$\\
         LAMP (\textbf{This Work}) & $\bm{\tO(\|C\|_\infty nm\varepsilon^{-1})}$ & $\bm{\mathcal{O}(n+m)}$\\
         \hline
    \end{tabular}
\end{table}
\footnotetext[5]{Here we specifically mean the space complexity to store iterates, not the input cost matrix $C$. In many settings, such as color transfer or Wasserstein distance estimation, distances can be computed on-the-fly using computational kernels, removing the need for an explicit $C$.}
\footnotetext[6]{Our PDOT complexity statement follows from combining Theorems 1 and 2 of~\cite{luPDOTPracticalPrimalDual2024}, which differs from the author's reported complexity in Corollary 1.}

The majority of first-order methods research targets EOT rather than directly solving OT owing to the smoothness of its dual function. In the EOT family, the Sinkhorn-Knopp (a.k.a. Sinkhorn) algorithm~\cite{sinkhornConcerningNonnegativeMatrices1967} serves as the standard method: conceptually simple, highly parallel, and rapidly convergent to low-accuracy solutions. Despite its advantages, the iteration complexity of Sinkhorn scales $\tO(\varepsilon^{-2})$~\cite{dvurechenskyComputationalOptimalTransport2018}, which results in slow convergence to high-accuracy solutions. In response, a number of alternative methods relying on classical ideas from first-order methods in optimization have been proposed, including 
accelerated primal-dual methods~\cite{dvurechenskyComputationalOptimalTransport2018,guminovCombinationAlternatingMinimization2021,linEfficiencyEntropicRegularized2022, luoImprovedRateFirst2023}, 
saddle-point methods~\cite{chambolleAcceleratedBregmanPrimalDual2022}, 
and variations of Sinkhorn~\cite{altschulerNearlinearTimeApproximation2017,linEfficiencyEntropicRegularized2022}. Additionally, several methods have been proposed to improve the empirical performance of EOT algorithms, including temperature annealing~\cite{schmitzerStabilizedSparseScaling2019a,kemertasTruncatedNewtonMethod2024,kemertasEfficientAccurateOptimal2025}, and second-order subroutines, such as truncated Newton~\cite{kemertasTruncatedNewtonMethod2024,kemertasEfficientAccurateOptimal2025}, and sparse Newton~\cite{wu2025pins,pan2026inexact} iterations. 

Other methods bypass EOT to target OT directly. Douglas-Rachford splitting~\cite{maiFastAccurateSplitting2021a}, dual extrapolation~\cite{jambulapatiDirectTildelbraceOrbrace2019a,assadiSemiStreamingBipartiteMatching2022}, and PDHG~\cite{luPDOTPracticalPrimalDual2024} are among the ``direct'' OT methods that have been proposed.

Table~\ref{tab:summary} catalogues a sample of the proposed first-order OT methods according to their computational complexity (relative to problem size $n$, $m$ and accuracy $\varepsilon$) and their storage requirements. While non-exhaustive, the table captures the essence of current first-order OT solvers. Dual-only methods such as Sinkhorn, Greenkhorn, and Accelerated Sinkhorn have favorable scaling in the problem dimension $n$, $m$ and attain linear-space complexity, however they have inferior worst-case dependence on the accuracy $\varepsilon$ and the cost scale $\|C\|_\infty$. The dual-only Dual Extrapolation implementation of~\cite{assadiSemiStreamingBipartiteMatching2022} manages to achieve state-of-the-art theoretical guarantees in linear space, however the underlying Dual Extrapolation method is empirically slow and the linear space implementation is difficult to implement in HPC environments, as discussed in Section~\ref{sec:discussion}.

\section{Dual Problems and Equivalence}\label{appdx:dual}
In this section we discuss the dual functions corresponding to the OT~\eqref{prob:primal_ot}, EOT~\eqref{prob:primal_eot}, and saddle-point~\eqref{prob:spp} problems. These functions then define the dual problems, which are useful theoretically (as we show in the next section), as well as providing a computable primal-dual gap for numerical implementation.

The dual problems for OT/EOT are obtained by first dualizing the constraints by standard Lagrangian machinery
\begin{equation}\label{prob:lagrange_ot}
    \max_{\varphi\in\R^n,\psi\in\R^m}\min_{X\in\Delta^{n\times m}}\inner{C}{X}+\inner{\row(X)-r}{\varphi} + \inner{\col(X)-c}{\psi} -\eta  H(X),
\end{equation}
where $\varphi$ and $\psi$ are the dual multipliers corresponding to the row and column constraints, respectively. If $\eta =0$, then minimizing with respect to $X$ over the simplex gives the nonsmooth problem
\begin{equation}\label{prob:dual_ot}
   \max_{\varphi\in\R^n,\psi\in\R^m}\left\{d(\varphi,\psi):=\min_{ij}\{C_{ij}+\varphi_i+\psi_j\}-\inner{r}{\varphi} - \inner{c}{\psi}\right\}.
\end{equation}
If $\eta >0$,  then minimizing with respect to $X$ over the simplex gives
\begin{equation}\label{def:dual_eot}
    \max_{\varphi\in\R^n,\psi\in\R^m} \left\{d^{\eta }(\varphi,\psi):=\softmin^{\eta }_{ij}\{C_{ij}+\varphi_i+\psi_j\}-\inner{r}{\varphi} - \inner{c}{\psi}\right\},
\end{equation}
where \[\softmin^{\eta }_{ij}\{C_{ij}+\varphi_i+\psi_j\}:=-\eta \log\biggl[\sum_{ij}\exp(-\eta ^{-1}(C_{ij}+\varphi_i+\psi_j))\biggr]\]
is the ``softmin'' function and is $\eta ^{-1}$-smooth with respect to the $\ell_2$ and $\ell_\infty$ norms~\cite[Example 5.15]{beckFirstOrderMethodsOptimization2017}.

Similarly, optimizing the primal variables $p$ in~\eqref{prob:spp} over the set of row-stochastic matrices gives the dual problem:
\begin{equation}\label{def:dual_spp}
    \max_{\theta\in[-1,1]^m}\biggl\{D(\theta)=\sum_{i=1}^nr_i\min_j\{C_{ij}+2\|C\|_\infty\theta_j\}-2\|C\|_\infty\inner{c}{\theta}\}\biggr\}, \quad \mbox{ if }  \eta =0, 
\end{equation}
 and 
\begin{equation}\label{def:dual_spp_e}
    \max_{\theta\in[-1,1]^m}\biggl\{D^{\eta }(\theta):=\sum_{i=1}^nr_i\softmin^{\eta }_j\{C_{ij}+2\|C\|_\infty\theta_j\}-2\|C\|_\infty\inner{c}{\theta}\}\biggr\}, \quad \mbox{ if }  \eta >0.
\end{equation}

\section{Deferred Proofs}\label{appdx:deferred}

First, we give a statement of the ``rounding'' from~\cite{altschulerNearlinearTimeApproximation2017} in Algorithm~\ref{alg:round}. ``Round'' takes as input a matrix $X\in\Rp^{n\times m}$ and marginals $r\in\Delta^n$, $c\in\Delta^m$ and returns a matrix $\tilde X$ satisfying $\row(\tilde{X})=r$, $\col(\tilde X)=c$. Note that in our statement we exchange the role of $r$ and $c$ compared to~\cite{altschulerNearlinearTimeApproximation2017}. However, the algorithms are clearly equivalent after transposing the input, which does not affect the guarantees.

\begin{algorithm}[H]\caption{Round{~\cite[Algorithm 2]{altschulerNearlinearTimeApproximation2017}}}\label{alg:round}
\begin{algorithmic}
    \REQUIRE $X\in\Rp^{n\times m}$, $r\in\Delta^n$, $c\in\Delta^m$.
    \STATE \textbf{Step 1)} Set $X'=X\diag{y}$ where
    $y_j=\min\left\{\frac{c_j}{\col(X)_j}, 1\right\}$.
    \STATE \textbf{Step 2)} Set $X''=\diag{x}X'$ where
    $x_i=\min\left\{\frac{r_i}{\row(X')_i}, 1\right\}$.
    \STATE \textbf{Step 3)} Compute
    $\delta_r=r-\row(X''),\quad \delta_c=c-\col(X'')$.
    \STATE \textbf{Step 4)} Set $\tilde X=X''+\|\delta_c\|_1^{-1}\delta_r\delta_c^\top$.
    \STATE \textbf{return} $\tilde X$.
\end{algorithmic}
\end{algorithm}

The following lemma, which is standard in OT literature, provides a useful property of Algorithm~\ref{alg:round}. {The version below is proved in~\cite{altschulerNearlinearTimeApproximation2017}, however the claim is usually weakened in the literature to treat $r$ and $c$ symmetrically.}

\begin{lemma}[{\cite[Lemma 7]{altschulerNearlinearTimeApproximation2017}}]\label{lem:jason_rounding}
    If $r\in\Delta^n$, $c\in\Delta^m$, and $X\in\Rp^{n\times m}$, then Algorithm~\ref{alg:round} takes $\mathcal{O}(nm)$ time to output a matrix $\tilde X\in \coup(r,c)$ satisfying
    \begin{equation}\label{ineq:l1_normbound_jason}
        \|X-\tilde X\|_1\leq 2\|\row(X)-r\|_1+\|\col(X)-c\|_1.
    \end{equation}
\end{lemma}

For simplicity, we refer the reader to~\cite{altschulerNearlinearTimeApproximation2017} for the the proof. We simply note that the authors prove the bound~\cite[Equation 6]{altschulerNearlinearTimeApproximation2017}
\begin{equation}
    \|\bar X - X\|_1\leq \|\row(X)-r\|_1 + 2\|\col(X)-c\|_1,
\end{equation}
where in their statement of Algorithm~\ref{alg:round} Steps \textbf{1} and \textbf{2} are reversed (first row scaling is applied, then column scaling). Our result follows from Equation 6 and applying the Round algorithm from~\cite{altschulerNearlinearTimeApproximation2017} to $X^\top$ and interchanging $r$ and $c$.

Next, we state several technical lemmas which will be used in the proof of Lemma~\ref{lem:l1_prob_equiv}.

The following lemma states a general property of the dual problem~\eqref{def:dual_eot}, which allows for constant-offset transformations in each dual variable.

\begin{lemma}\label{lem:logsumexp_additive}
    Let $\varphi\in\R^n$, $\psi\in \R^m $ be two dual potentials. Then, for any $a\in\R$, $b\in\R$,
    \[
    d^{\eta }(\varphi + a\one_n,\psi+b\one_m)=d^{\eta }(\varphi,\psi),
    \]
    where $d^{\eta }$ is the EOT dual defined in~\eqref{def:dual_eot}.
\end{lemma}
\begin{proof}
    We observe that the $\LSE$ function is directionally affine along the $\one_n$-axis, i.e.,
    \[\LSE(x+\alpha\one_n) = \log\left[\sum_{i=1}^n\exp\left(x_i+\alpha\right)\right]=\LSE(x)+\alpha,\]
    for any $x\in\R^n$ and $\alpha\in\R$.
    Then, we have
    \[
    \begin{split}
        d^{\eta }(\varphi + a\one_n,\psi+b\one_m)=&-\inner{\varphi}{r} - a -\inner{\psi}{c}-b\\
        &-\eta \LSE\left(-\eta ^{-1}(C + \varphi\one_m^\top + \one_m\psi^\top)\right) + a + b=d^{\eta }(\varphi,\psi),
    \end{split}
    \]
    which proves the claim.
\end{proof}
The next result is a technical lemma regarding exact penalty formulations.
\begin{lemma}\label{lem:affine_minimizer}
    Consider the linearly constrained problem
    \begin{equation}\label{eqn:unpen_general}
    \min_{x\in Q} f(x)\quad\text{s.t.}\quad Ax=b,
    \end{equation}
    where $f:\R^n\to(-\infty,\infty]$ is a closed proper and convex function, $A\in\R^{m\times n}$, $b\in\R^m$, and $Q$ is a closed and convex set. Further assume that there exists a point $x\in\operatorname{relint}(\operatorname{dom}f)\cap \operatorname{relint}(Q)$ satisfying $Ax=b$, where $\operatorname{relint}(\operatorname{dom}f)$ is the relative interior of the domain of $f$ (i.e., Slater's condition is satisfied).
    
    Define the exact $\ell_1$ penalization as
    \begin{equation}\label{eqn:l1_pen_general}
    \min_{x\in Q} \{\phi_\mu(x)=f(x) + \mu\|Ax-b\|_1\}.
    \end{equation}
    Suppose $x^*$ is a solution of~\eqref{eqn:unpen_general} with Lagrange multiplier $\lambda^*\in\R^m$. Then, for any $\mu\geq\|\lambda^*\|_\infty$, we have $x^*\in\Argmin_{x\in Q}\phi_\mu(x)$. 
\end{lemma}
\begin{proof}
    Consider the saddle-point form of~\eqref{eqn:unpen_general} formed from the Lagrangian
    \begin{equation}\label{def:pen_spp}
         \min_{x\in Q}\max_{\lambda\in \R^m}\{\mathcal{L}(x,\lambda)= f(x)+\inner{\lambda}{Ax-b}\}.
    \end{equation}
    By standard arguments, Slater's condition implies that the pair $(x^*,\lambda^*)$ is a saddle-point of~\eqref{def:pen_spp}. By Hölder's inequality and the saddle-point property, for any $x\in Q$, we have
    \begin{align*}
        f(x^*)&= \mathcal{L}(x^*,\lambda^*)\leq  \mathcal{L}(x,\lambda^*)
        =f(x)+\inner{\lambda^*}{Ax-b}
        \leq f(x)+\|\lambda^*\|_\infty\|Ax-b\|_1\leq \phi_\mu(x).
    \end{align*}
    Since $f(x^*)=\phi_\mu(x^*)$, we have $x^*\in\Argmin_{x\in Q}\phi_\mu(x)$ as claimed.
\end{proof}

Therefore, if we can guarantee that $\|\psi^*\|_\infty\leq 2\|C\|_\infty$ for some optimal multiplier $\psi^*\in\R^m$, we can guarantee that~\eqref{prob:primal_eot} and~\eqref{prob:l1_ot_primal} have the same minimizer (unique minimizer, for $\eta>0$). Fortunately, we have the following bound from~\cite{linEfficiencyEntropicRegularized2022} characterizing the infinity norm of the dual multipliers. The proof follows~\cite[Lemma 3]{linEfficiencyEntropicRegularized2022}, however we show the argument here since our statement differs slightly from theirs, as they use a common constant rather than providing separate bounds for $\|\varphi^*\|_\infty$ and $\|\psi^*\|_\infty$.
\begin{lemma}\label{lem:dual_infty_bound}
    For the dual EOT problem defined in~\eqref{def:dual_eot}, there exists an optimal solution $(\varphi^*,\psi^*)$ such that
    \[
        \|\varphi^*\|_\infty\leq \|C\|_\infty -\eta \log\underset{1\leq i\leq n}\min\{r_i\}, \quad 
        \;\|\psi^*\|_\infty\leq \|C\|_\infty -\eta \log\underset{1\leq j\leq m}\min\{c_j\}.
    \]
\end{lemma}
\begin{proof}
    We first claim that there exist dual variables $\varphi^*$ and $\psi^*$ which are optimal solutions to~\eqref{def:dual_eot} satisfying
    \begin{equation}\label{ineq:comp_bounds}
        \nonumber\min_i\varphi^*_i\leq 0\leq \max_i\varphi^*_i, \quad 
        \min_j\psi^*_j\leq 0\leq \max_j\psi^*_j.
    \end{equation}
    Let $\hat\varphi^*$, $\hat\psi^*$ be a pair of optimal dual potentials. Define \[
        \Delta_\varphi=\frac{1}{2}(\max_i\hat\varphi_i^*+\min_i\hat\varphi_i^*),\quad
        \Delta_\psi=\frac{1}{2}(\max_j\hat\psi_j^*+\min_j\hat\psi_j^*),
    \]
    and set $\varphi^*=\hat\varphi^* - \Delta_\varphi\one_n$ and $\psi^*=\hat\psi^* - \Delta_\psi\one_m$.
    By Lemma~\ref{lem:logsumexp_additive}, $\varphi^*$ and $\psi^*$ are optimal dual potentials. Furthermore, by construction, they satisfy~\eqref{ineq:comp_bounds}.
    Now, taking the gradient of the dual objective in~\eqref{def:dual_eot}, we have
    \begin{equation}\label{eq:grad_dual_opt}
        0=-r_i +\exp[-\eta ^{-1}\varphi^*_i]\sum_{j=1}^m\frac{\exp[-\eta ^{-1}(C_{ij}+\psi^*_j)]}{Z},
    \end{equation}
    where
    \[Z=\sum_{k,\ell}\exp[-\eta ^{-1}(C_{k\ell}+\varphi^*_k+\psi^*_\ell)].\]
    Rearranging~\eqref{eq:grad_dual_opt}, we obtain
    \begin{equation*}
        \varphi^*_i=-\eta \log r_i + \eta \log\left(\sum_{j=1}^m\exp[-\eta ^{-1}(C_{ij}+\psi^*_j)]\right) - \eta \log Z.
    \end{equation*}
    Since each $C_{ij}\in[0, \|C\|_\infty]$ we have for all $i$, $j$
    \begin{equation*}
    \exp[-\eta ^{-1}(C_{ij}+\psi^*_j)]\geq \exp[-\eta ^{-1}(\psi^*_j)]\exp[-\eta ^{-1}\|C\|_\infty].
    \end{equation*}
    We also have $-\log r_i\geq 0$, hence
    \begin{equation*}
        \varphi^*_i\geq  \eta \log\left(\sum_{j=1}^m\exp[-\eta ^{-1}\psi^*_j]\right) -\|C\|_\infty - \eta \log Z.
    \end{equation*}
    Again using $C_{ij}\geq 0$, we have
    \begin{equation*}
        \varphi^*_i\leq -\eta \log r_i + \eta \log\left(\sum_{j=1}^m\exp[-\eta ^{-1}\psi^*_j]\right) - \eta \log Z.
    \end{equation*}
    Then we have
    \begin{align*}
        \max_{i}\varphi_i^*&\leq -\eta \min_i\log r_i+\eta \log\left(\sum_{j=1}^m\exp[-\eta ^{-1}\psi^*_j]\right)-\eta \log Z\\
        \min_{i}\varphi_i^*&\geq  \eta \log\left(\sum_{j=1}^m\exp[-\eta ^{-1}\psi^*_j]\right) - \|C\|_\infty - \eta \log Z.
    \end{align*}
    So,
    \begin{align*}
        \max_{i}\varphi_i^*-\min_{i}\varphi_i^*\leq \|C\|_\infty-\eta \log\min_{i} r_i.
    \end{align*}
    Since $\min_{i}\varphi_i^*\leq 0$ and $\max_{i}\varphi_i^*\geq 0$, we have that 
    \[
        0\geq\min_{i}\varphi_i^*\geq -(\|C\|_\infty-\eta \log\min_{i} r_i), \quad
        0\leq\max_{i}\varphi_i^*\leq \|C\|_\infty-\eta \log\min_{i} r_i,
    \]
    which implies that 
    \[
    \|\varphi^*\|_\infty \leq \|C\|_\infty-\eta \log\min_{i} r_i
    \]
    as claimed. A similar argument holds for the dual variable $\psi^*$ with the $c$ marginal in place of the $r$ marginal.  
\end{proof}
A similar result can be shown for unregularized OT with bounded costs, see~\cite[Remark 1.13]{villaniTopicsOptimalTransportation2003}.

\vspace{1em}
We now state the full equivalence result extending~\cite[Lemma 2.3]{jambulapatiDirectTildelbraceOrbrace2019a} referenced in Subsection~\ref{subsec:saddle}.

\begin{lemma}\label{lem:l1_prob_equiv}
    Given $\varepsilon>0$ and $\eta \geq 0$, let $r\in\Delta^n$, $c\in \Delta^m$, and $C\in\Rp^{n\times m}$ be the input for the (E)OT problem, and assume that $r$ and $c$ have no zero entries. Then, we have:
    \begin{itemize}
        \item[{a)}] if $\eta =0$ and $p$ is an $\varepsilon$-solution to~\eqref{prob:l1_ot_primal}, then we can map $p$ to an $\varepsilon$-solution of~\eqref{prob:primal_ot} in $\mathcal{O}(nm)$ arithmetic operations;
        \item[{b)}] if $\eta \in(0,\|C\|_\infty/\log(\max_{j}c_j^{-1})]$, then the unique minimizer $p^*$ of~\eqref{prob:l1_ot_primal} is equivalent to the unique minimizer $X^*$ of~\eqref{prob:primal_eot} under the mapping $X^*=\diag{r}p^*$.
    \end{itemize}
\end{lemma}
\begin{proof}
a) The claim directly follows from Lemma 2.3 of~\cite{jambulapatiDirectTildelbraceOrbrace2019a}, though we reproduce the proof here for completeness. Let $P(\diag{r}p)$ be the objective function of~\eqref{prob:l1_ot_primal} with $\eta =0$. First, we show that there exists an optimizer to~\eqref{prob:l1_ot_primal} which is column feasible. Let $p^*\in\{\Delta^m\}^n$ be an arbitrary optimizer to~\eqref{prob:l1_ot_primal} and denote $\diag{r}\tilde p^*=\operatorname{Round}(\diag{r}p^*, r, c)$.
By Lemma~\ref{lem:jason_rounding}, we have $\diag{r}\tilde p^*\in\coup(r,c)$ and the bound
\begin{equation}\label{ineq:l1_normbound}
    \|\diag{r}\tilde p^* - \diag{r}p^*\|_1\leq 2\|\col_r(p^*) - c\|_1.
\end{equation}

Since $\diag{r}\tilde p^*$ is column feasible (hence $\|\col_r(\tilde p^*)-c\|_1=0$), we have
\begin{align*}
    P(\diag{r}\tilde p^*)-P(\diag{r}p^*) &= \inner{C}{\diag{r}(\tilde p^*-p^*)}-2\|C\|_\infty\|\col_r(p^*) - c\|_1.
\end{align*}
Applying H\"{o}lder's inequality, we have
\begin{align*}
    P(\diag{r}\tilde p^*)&-P(\diag{r}p^*)\leq \|C\|_\infty\|\diag{r}\tilde p^*-\diag{r}p^*\|_1-2\|C\|_\infty\|\col_r(p^*) - c\|_1\\
    &\stackrel{\eqref{ineq:l1_normbound}}{\leq} 2\|C\|_\infty\|\col_r(p^*) - c\|_1-2\|C\|_\infty\|\col_r(p^*) - c\|_1=0.
\end{align*}
Then $P(\diag{r}\tilde p^*)-P(\diag{r}p^*)\leq 0$, which implies that $P(\diag{r}\tilde p^*)$ is also optimal. Since $\coup(r,c)\subseteq\{\diag{r}p:p\in\{\Delta^m\}^n\}$, we have
\[P(\diag{r}\tilde p^*)=\min_{p\in\{\Delta^m\}^n}P(\diag{r}p)\leq \min_{X\in\coup(r,c)}\inner{C}{X}.\]
Since $\diag{r}\tilde p^*:=\tilde X^*\in \coup(r,c)$, we clearly have $\min_{X\in\coup(r,c)}\inner{C}{X}\leq \inner{C}{\tilde X^*}$, hence $\inner{C}{\tilde X^*}=\min_{X\in\coup(r,c)}\inner{C}{X}$
and $\tilde X^*$ is therefore an optimizer to~\eqref{prob:primal_ot}.

Now, let $p^*\in\{\Delta^m\}^n$ be an optimizer to~\eqref{prob:l1_ot_primal} satisfying $\col_r(p^*)=c$ (whose existence we have just proved). Let $p$ be an $\varepsilon$-solution to~\eqref{prob:l1_ot_primal} and $\diag{r} \tilde p=\operatorname{Round}(\diag{r}p,r,c)$. Then
\begin{align*}
    \varepsilon \geq P(\diag{r}p)-P(\diag{r}p^*)=P(\diag{r}p)-P(\diag{r}\tilde{p})+P(\diag{r}\tilde p)-P(\diag{r}p^*).
\end{align*}
We have the lower bound
\begin{align*}
    P(\diag{r}p)-P(\diag{r}\tilde p)&=\inner{C}{\diag{r}(p-\tilde{p})}+2\|C\|_\infty\|\col_r(p)-c\|_1\\
    &\geq -\|C\|_\infty\|{\diag{r}(p-\tilde{p})}\|_1+2\|C\|_\infty\|\col_r(p)-c\|_1\stackrel{\eqref{ineq:l1_normbound_jason}}\geq 0,
\end{align*}
where the first inequality follows by Hölder's inequality and the second by Lemma~\ref{lem:jason_rounding}.
Then we have
\begin{align*}
     P(\diag{r}\tilde p)-P(\diag{r}p^*) \le \varepsilon,
\end{align*}
proving the claim.

b) 
Define $\psi^*$ as some optimal Lagrange multipliers (i.e., part of an optimal pair $(\varphi^*, \psi^*)$ for~\eqref{def:dual_eot}) for the $c$-marginal constraint. We then rewrite the constrained problem~\eqref{prob:primal_eot} as $\min_{p\in Q} f(p)\text{ s.t. }Ap=c$, where
\[
f(p)=\inner{C}{\diag{r}p}-\eta  H(\diag{r}p),
\] 

$A$ is the linear mapping implementing the column sum $Ap:=\col_r(p)$, and $Q=\{\Delta^m\}^n$. Let $p^*\in\Argmin_{p\in Q} f(p)\text{ s.t. }Ap=c.$ 

Then, by Lemma~\ref{lem:affine_minimizer} (noting that Slater's condition is satisfied for OT problems, since the product coupling is feasible), choosing the penalty coefficient $\mu\geq \|\psi^*\|_\infty$ implies $p^*\in\Argmin_{p\in Q}\{f(p)+\mu\|Ap-c\|_1\}$. Since $r$ has full support, $-H(\diag{r}p)$ (and therefore $f$) is strongly convex on $\{\Delta^m\}^n$~\cite[Example 5.27]{beckFirstOrderMethodsOptimization2017}, which implies that $p^*$ is the unique minimizer of both the penalized and constrained problems. 

Therefore, if $\|\psi^*\|_\infty\leq 2\|C\|_\infty$,~\eqref{prob:primal_eot} and~\eqref{prob:l1_ot_primal} have the same unique minimizer. Applying the bound from Lemma~\ref{lem:dual_infty_bound}, we have the condition
\[
    \|C\|_\infty - \eta \log\underset{j}\min\{c_j\}\leq 2\|C\|_\infty, \qquad \mbox{or} \qquad 
    -\eta \log\min_j\{c_j\}\leq \|C\|_\infty.
\]
The condition on $\eta $ follows from rearranging.
\end{proof}

Next, we prove the closed-form solutions for the primal and dual mirror maps introduced in Subsection~\ref{sec:mirror_maps}.

\noindent
\textbf{Proof of Equation~\eqref{def:closed_form_opt}: } 
Beginning with the primal mirror map, we have
\begin{equation*}
\begin{split}
    p^+=&\mathcal{M}_\tau^{\eta }(p^0;\theta)\\
    \stackrel{\eqref{def:mirror_maps}}=&\underset{p\in\{\Delta^m\}^n}\argmin\biggl\{\inner{\diag{r}(C+\eta (\log p^0 + \one_n\one_m^\top) + 2\|C\|_\infty\theta) }{p-p^0} +\frac{1}{\tau}\KL(\diag{r}p\|\diag{r}p^0)\biggr\}.
    \end{split}
\end{equation*}
Writing the optimality conditions for each $p^+_{ij}$, we have
\begin{equation*}
    0 = \tau r_i(C_{ij} + 2\|C\|_\infty\theta_j + \eta ) + r_i\eta \tau\log p^0_{ij} + r_i\log \frac{p^+_{ij}}{p^0_{ij}} + \lambda_i,
\end{equation*}
where the $\lambda_i$ terms enforce row normalization. Solving for $p$ gives
\begin{equation*}
    p^+_{ij} = (p^0)_{ij}^{1-\tau\eta }\exp\left[-\tau(C_{ij}+2\|C\|_\infty\theta_j)-\log(Z_i)\right],
\end{equation*}
where the log-normalization terms $\log Z_i$ absorb all row constants. We therefore obtain the form in the first part of~\eqref{def:closed_form_opt}.

We now prove the closed form solution for the dual mirror map. By definition, we have
\begin{equation*}
    \theta^+=\mathcal{M}_\tau^{\eta }(\theta^0;p)\stackrel{\eqref{def:mirror_maps}}=\argmax_{\theta\in\R^m}\left\{\inner{2\|C\|_\infty(\col_r(p)-c)}{\theta-\theta^0}-\frac{1}{\tau}\D_{H_c^\alpha}(\theta\|\theta^0)\right\}.
\end{equation*}
Then for each index $j\in\{1,\dots,m\}$ we have the optimality condition
\begin{equation*}
    0=2\tau\|C\|_\infty(\col_r(p)_j-c_j) - \frac{c^\alpha_j}{2}\log\left(\frac{\theta_j^++1}{1-\theta_j^+}\frac{1-\theta_j^0}{1+\theta_j^0}\right)
\end{equation*}
which gives
\begin{equation*}
    \frac{\theta^+_j+1}{1-\theta^+_j}= \frac{1+\theta_j^0}{1-\theta_j^0}\exp\left[4\tau\|C\|_\infty\frac{\col_r(p)_j-c_j}{c^\alpha_j}\right].
\end{equation*}
After some algebraic manipulation, we obtain
\begin{equation*}
    \theta^+_j= \frac{\left(\frac{1+\theta_j^0}{1-\theta_j^0}\right)e^{4\tau\|C\|_\infty(\col_r(p)_j-c_j)/c^\alpha_j} - 1}{\left(\frac{1+\theta_j^0}{1-\theta_j^0}\right)e^{4\tau\|C\|_\infty(\col_r(p)_j-c_j)/c^\alpha_j} + 1}=\tanh\left[\frac{2\tau\|C\|_\infty}{c^\alpha_j}(\col_r(p)_j-c_j)+\frac{1}{2}\log\frac{1+\theta_j^0}{1-\theta_j^0}\right],
\end{equation*}
as claimed, where the second equality follows from $\tanh(x)=(e^{2x}-1)/(e^{2x}+1)$.\qed

\noindent\textbf{Proof of Lemma~\ref{lem:primal_md_eq_dual_avg}}
    By our assumption on $\eta$ and $\gamma$, it follows that $\tau\eta'\leq 1$, so $\theta'$ is a convex combination of $\theta^a$ and $\theta^b$, and so $\theta'\in[-1,1]^m$. 
    
    It then follows from simple algebra and the definition of $\eta'$ that $(1-\tau\eta)\gamma^{-1} = (\eta')^{-1} - \tau$. Then, using the definition of the dual-to-primal map, we have
    \begin{align}
        p^\gamma(\theta^a))^{1-\tau\eta}_{ij}&\stackrel{\eqref{def:dual_to_primal}}\propto\exp[-(1-\tau\eta)\gamma^{-1}(C_{ij}+2\|C\|_\infty\theta_j^a)] \nonumber\\
        &=\exp\left[-\left(\frac1{\eta'}-\tau\right)(C_{ij}+2\|C\|_\infty\theta_j^a)\right].\label{eq:eta_algebra_pd}
    \end{align}

    Then, using the primal mirror map $\mathcal{M}^\gamma_{\tau}(\cdot;\theta)$ in \eqref{def:closed_form_opt}, we obtain
    \begin{align*}
        \mathcal{M}^\eta_{\tau}(p^\gamma(\theta^a);\theta^b)_{ij} &\stackrel{\eqref{def:closed_form_opt}}\propto(p^\gamma(\theta^a))^{1-\tau\eta}_{ij}\exp\left[-\tau(C_{ij}+2\|C\|_\infty\theta^b_j)\right]\\
        &\stackrel{\eqref{eq:eta_algebra_pd}}\propto\exp\left[-\left(\frac1{\eta'}-\tau\right)(C_{ij}+2\|C\|_\infty\theta_j^a) -\tau(C_{ij}+2\|C\|_\infty\theta^b_j)\right]\\
        &=\exp\left[-\frac1{\eta'}(C_{ij}+2\|C\|_\infty\theta_j')\right]\stackrel{\eqref{def:dual_to_primal}}\propto p^{\eta'}(\theta')_{ij},
    \end{align*}
    where the last identity follows by the definition of $\theta'$.\qed\\
\noindent\textbf{Proof of Proposition~\ref{prop:equivalence}}
    We prove the claim by induction. First, we consider the base case $t=0$. Since $\eta_0=\infty$ and $\theta_t=\nu_t=\zero_m$, we have
    \begin{equation*}
        p^{\eta_0}(\nu^0)_{ij}\propto\exp[-\eta_0^{-1}(C_{ij}+2\|C\|_\infty \nu_j)]=1,
    \end{equation*}
    therefore, after row-wise normalization, $p^{\eta_0}(\nu^0)_{ij}=1/m=p^0_{ij}$.

    For the inductive step, suppose that the claim holds for iteration $t\geq 0$, i.e., $p^t=p^{\eta_t}(\nu^t)$. Then, using Lemma~\ref{lem:primal_md_eq_dual_avg} with $(\eta ,\gamma,\tau,\theta^a,\theta^b)=(\eta ,\eta_t,\tau_1,\nu^t,\theta^t)$, and the choice of $\eta_{t+1}$ in~\eqref{step:eta_update}, we have
    
    \begin{equation*}
        \bar p^{t+1}\stackrel{\eqref{step:pdmp_midpoint}}=\mathcal{M}_{\tau_1}^{\eta }(p^t;\theta^t)\stackrel{\eqref{eq:mirror_map_id},\eqref{step:lamp_mid_nu}}=p^{\eta_{t+1}}(\bar\nu^{t+1}).
    \end{equation*}
    Similarly, using Lemma~\ref{lem:primal_md_eq_dual_avg} with $(\eta ,\gamma,\tau,\theta^a,\theta^b)=(\eta ,\eta_t,\tau_1,\nu^t,\bar\theta^{t+1})$, and the choice of $\eta_{t+1}$ in~\eqref{step:eta_update} we have
    \begin{equation*}
        p^{t+1}\stackrel{\eqref{step:pdmp_main}}=\mathcal{M}_{\tau_1}^{\eta }(p^t;\bar\theta^{t+1})\stackrel{\eqref{eq:mirror_map_id},\eqref{step:lamp_main_nu}}=p^{\eta_{t+1}}(\nu^{t+1}),
    \end{equation*}
    therefore the claim holds for all $t\geq 0$.\qed
\section{Analysis of Algorithm~\ref{alg:lamp}}\label{appdx:analysis}

In this section, we analyze Algorithm~\ref{alg:lamp} in the performant parameter regime ($\eta =0$, $\beta=\log 3$, $\tau_1=\tau_2=1/(2\|C\|_\infty)$). Under this setting, we have $\eta_0=\infty$ and $\eta_t=2\|C\|_\infty/t$ for $t\geq 1$.

For convenience, we restate the LAMP algorithm with our empirical parameter choices in Algorithm~\ref{alg:lamp_appdx}, where we use $\tanh(\log(3)/2)=1/2$.
\begin{algorithm}[H]\caption{LAMP (Performant Parameters)}\label{alg:lamp_appdx}
    \begin{algorithmic}
        \REQUIRE $r\in\Delta^n$, $c\in\Delta^m$, $\alpha >0$, set $\nu_0=\theta_0=\zero_m$.
        \FOR{$t\geq 0$}
        \STATE Set $\eta_t=2\|C\|_\infty/t$ (or $\eta_t=\infty$ if $t=0$) and compute the midpoints
        \begin{align}
        \bar\nu^{t+1} &= \frac{t}{t+1}\nu^t + \frac{1}{t+1}\theta^t\label{def:bar_nu}\\
            \bar \theta^{t+1} &= \argmax_{\theta\in \R^m}\{\inner{\col_r(p^{\eta_t}(\nu^t))-c}{\theta} - \DHa(\theta\|\theta^t)\}\label{def:bar_theta}
        \end{align}
        \STATE Set $\eta_{t+1}=2\|C\|_\infty/(t+1)$ and compute the main sequences
        \begin{align}
         \nu^{t+1} &= \frac{t}{t+1}\nu^t + \frac{1}{t+1}\bar\theta^{t+1}\label{def:nu}\\
            \hat\theta^{t+1} &= \argmax_{\theta\in \R^m}\{\inner{\col_r(p^{\eta_{t+1}}(\bar\nu^{t+1}))-c}{\theta} - \DHa(\theta\|\theta^t)\}\label{def:theta_hat}\\
            \theta^{t+1}&=\operatorname{clip}(\hat\theta^{t+1}, -1/2,1/2).\label{def:theta}
        \end{align}
        \ENDFOR
    \end{algorithmic}
\end{algorithm}
Note that by our initialization $\theta_0=\zero_m$ and the definition of the dual mirror map in~\eqref{def:closed_form_opt}, we have $\bar\theta^{t+1}\in(-1,1)^m$ and $\hat\theta^{t+1}\in(-1,1)^m$ for all $t\geq 0$. Since $\bar\nu^{t+1}$ and $\nu^{t+1}$ are convex combinations of the $\{\theta^{t}\}$ and $\{\bar\theta_t\}$ sequences, we then have $\bar\nu^{t+1}\in(-1,1)^m$ and $\nu^{t+1}\in(-1,1)^m$ for all $t\geq 0$.

With our parameter choices, the dual-to-primal map on each iteration $t\geq 1$ is
\begin{equation}\label{def:pd_map_appdx}
p^{\eta_t}(\nu)_{ij}\propto\exp\left[-t\left(\frac{C_{ij}}{2\|C\|_\infty}+\nu_j\right)\right].
\end{equation}
We therefore observe that the choice $\tau_1=\tau_2=1/(2\|C\|_\infty)$ effectively normalizes and halves $C$ in all step computations, hence for simplicity of notation we assume that $\|C\|_\infty=1$ for intermediate results. 

We observe that the gradient of the dual entropy $H_c^\alpha(\nu)$ is
\begin{equation}\label{def:Ha_grad}
    \nabla H_c^\alpha(\nu)=\frac{1}{2}c^\alpha\odot{\log\frac{1+\nu}{1-\nu}},\footnote{We note that this can be simplified to $\nabla H_c^\alpha(\nu)={c^\alpha}\odot{\operatorname{arctanh}(\nu)}$, however this form is not explicitly used in our analysis.}
\end{equation}
where $\odot$ is an elementwise Hadamard product. Then, for $\theta\in[-1,1]^m$, $\nu\in(-1,1)^m$, the Bregman divergence $\DHa(\theta\|\nu)$ has the form
\begin{align}
    \DHa(\theta\|\nu)&=\inner{c^\alpha}{\frac{\theta+1}{2}\log\left[\frac{\theta+1}{\nu+1}\right] +\frac{1-\theta}{2}\log\left[\frac{1-\theta}{1-\nu}\right]}\label{def:DHa_explicit}\\
    &=\inner{\frac{c^\alpha}{2}}{\log\left[\frac{1-\theta^2}{1-\nu^2}\right]} +\inner{\theta}{\nabla H_c^\alpha(\theta)-\nabla H_c^\alpha(\nu)},\label{def:DHa}
\end{align}
where the second equality can be shown by algebra.

Furthermore, define the vector-valued function $Z^t:[-1,1]^m\to\R^n$ as
    \begin{equation}\label{def:normalizer}
        Z^t(\theta)_i=\sum_{j=1}^m\exp\left[-t\left(\frac{C_{ij}}{2\|C\|_\infty}+\theta_j\right)\right],
    \end{equation}
    which is the set of row normalization constants for $p^{\eta_t}(\theta)$.

For further convenience, we denote 
\begin{equation}\label{def:primal_subs}
\begin{gathered}
    X^t:=\diag{r}p^{\eta_t}(\nu^t),\,\, \bar X^{t+1}:=\diag{r}p^{\eta_{t+1}}(\bar\nu^{t+1});\\
    \col(X^t):=\col_r(p^{\eta_t}(\nu^t)),\,\, \col(\bar X^{t+1}):=\col_r(p^{\eta_{t+1}}(\bar \nu^{t+1})).
    \end{gathered}
\end{equation}
By definition, then, we have for all $t\geq 0$ that 
\begin{equation}\label{eq:row_feasible}
    \row(X^t)=\row(\diag{r}p^{\eta_t}(\nu^t))=r,\quad \row(\bar X^{t+1})=\row(\diag{r}p^{\eta_{t+1}}(\bar\nu^{t+1}))=r.
\end{equation}

For further notational convenience, we rewrite the unregularized ($\eta=0$) saddle-point problem~\eqref{prob:spp} as a problem over the coupling $X$ rather than the row-stochastic matrix $p$
\begin{equation}\label{prob:spp_unreg}
    \min_{X\in\Delta_r^{n\times m}}\max_{\theta\in[-1,1]^m}\left\{K(X,\theta):=\inner{C}{X}+2\|C\|_\infty\inner{\theta}{\col(X)-c}\right\},
\end{equation}
where $\Delta_r^{n\times m}=\{X\in\Delta^{n\times m}:\row(X)=r\}$ is the set of row-feasible couplings.

To begin, we show that there exists a dual optimizer in the set $[-1/2,1/2]^m$, making the clipping step in~\eqref{def:theta} benign.

\begin{lemma}\label{lem:b_choice}
    Suppose that $\eta=0$. There exists a saddle-point of~\eqref{prob:spp_unreg} $(X^*,\theta^*)$ where $\theta^*\in[-1/2, 1/2]^m$ and $X^*$ is a minimizer of~\eqref{prob:primal_ot}.
\end{lemma}
\begin{proof}
    We recall that, when $\|C\|_\infty<\infty$, there exists an optimal dual pair $\varphi^*\in\R^n$, $\psi^*\in\R^m$ to~\eqref{prob:dual_ot} satisfying $\|\varphi^*\|_\infty\leq \|C\|_\infty$, $\|\psi^*\|_\infty\leq \|C\|_\infty$ (see ~\cite[Remark 1.13]{villaniTopicsOptimalTransportation2003}). Let $X^*\in \coup(r,c)$ be the corresponding primal solution, which, by strong duality, is an optimizer to~\eqref{prob:primal_ot}. Then, note that the following saddle-point problems are equivalent
    \begin{align}
        \nonumber\min_{X\in\Delta^{n\times m}}\max_{\varphi\in\R^n,\psi\in\R^m}\{&\inner{C}{X}+\inner{\varphi}{\row(X)-r}+\inner{\psi}{\col(X)-c}\}\\
        =&\min_{X\in\Delta_r^{n\times m}}\max_{\psi\in\R^m}\left\{\inner{C}{X}+\inner{\psi}{\col(X)-c}\right\}\label{prob:spp_reform1}\\
        =&\min_{X\in\Delta_r^{n\times m}}\max_{\theta\in\R^m}\left\{K(X,\theta):=\inner{C}{X}+2\|C\|_\infty\inner{\theta}{\col(X)-c}\right\}\label{prob:spp_reform2}.
    \end{align}
      Define $\theta^*=\psi^*/(2\|C\|_\infty)\in[-1/2,1/2]^m$, which satisfies $\theta^*\in[-1/2,1/2]^m$.  
    Since $(X^*,\psi^*)$ is a saddle-point of~\eqref{prob:spp_reform1}, $(X^*,\theta^*)$ is a saddle-point of~\eqref{prob:spp_reform2}. Finally, by the saddle-point property, for all $X\in\Delta_r^{n\times m}$ and $\theta\in\R^m$, we have
    \[
        K(X^*,\theta)\leq K(X^*,\theta^*)\leq K(X,\theta^*).
    \]
    Since this is also true for all $\theta\in[-1,1]^m$, then $(X^*,\theta^*)$ is a saddle-point of~\eqref{prob:spp_unreg}, completing the proof.
\end{proof}

Next, we show that the clip operation in~\eqref{def:theta} can only decrease the divergence from another point in the subset $[-1/2, 1/2]^m$. 
\begin{lemma}\label{lem:clipping_divergence}
    Let $\theta\in [-1/2,1/2]^m$ and $\hat\nu\in(-1,1)^m$, and define
    \[
        \nu=\operatorname{clip}(\hat\nu, -1/2, 1/2).
    \]
     Then,
     \begin{equation}\label{ineq:clipping_divergence}
         \DHa(\theta\|\nu)\leq \DHa(\theta\|\hat\nu).
     \end{equation}
\end{lemma}
\begin{proof}
    We begin by recalling the optimality condition for a convex optimization problem $\min_{x\in\mathcal{X}}f(x)$  (see, e.g., ~\cite[Equation 4.21]{boyd2004convex})
    \begin{equation}\label{eq:constrained_opt}
        x^*=\argmin_{x\in\mathcal{X}}f(x)\text{ \textbf{if and only if} }x^*\in\mathcal{X} \text{ and } \inner{\nabla f(x^*)}{x-x^*}\geq 0\text{ for all } x\in \mathcal{X}. 
    \end{equation}
    Now, we claim that $\nu=\argmin_{\mu\in[-1/2, 1/2]^m}\DHa(\mu\|\hat\nu)$. Suppose for some $j$ that $\hat\nu_j>1/2$, so $\nu_j=1/2$. Then, consider the convex univariate function
    \begin{equation}
       \phi_j(\gamma)= \frac{1+\gamma}{2}\log\left(\frac{1+\gamma}{1+\hat\nu_j}\right)+\frac{1-\gamma}{2}\log\left(\frac{1-\gamma}{1-\hat\nu_j}\right).
    \end{equation}
    The function $\phi_j(\gamma)$ is nonnegative over $\gamma\in[-1,1]$, with its minimum at $\hat\nu_j$ (as one can verify). Therefore, if $\hat\nu_j\in[-1/2, 1/2]$, then $\nu_j=\hat\nu_j=\argmin_{\gamma\in[-1/2, 1/2]}\phi_j(\gamma)$. Now, assume that $\hat\nu_j\not\in[-1/2, 1/2]$.
    Taking the derivative of $\phi_j$, we obtain
    \begin{equation}
        \phi_j'(\gamma)=\frac{1}{2}\log\left(\frac{1+\gamma}{1-\gamma}\frac{1-\hat\nu_j}{1+\hat\nu_j}\right).
    \end{equation}
    Note that if $\hat\nu_j>1/2$, we have $1+\hat\nu_j>3/2$ and $1-\hat\nu_j<1/2$, so
    \begin{equation}
        \phi_j'(1/2)=\frac{1}{2}\log\left(\frac{3/2}{1/2}\frac{1-\hat\nu_j}{1+\hat\nu_j}\right)=\frac{1}{2}\log\left(\frac{3/2}{1+\hat\nu_j}\right)+\frac{1}{2}\log\left(\frac{1-\hat\nu_j}{1/2}\right)<0.
    \end{equation}
    Therefore, at $1/2$, $-\phi_j'(1/2)=-\phi_j'(\nu_j)$ and therefore $\phi_j'(\nu_j)(x-\nu_j)\geq 0$ for all $x\in[-1/2,1/2]$, so $\nu_j=\argmin_{\gamma\in[-1/2, 1/2]}\phi_j(\gamma)$ by~\eqref{eq:constrained_opt}. A symmetric argument shows that $\nu_j=-1/2=\argmin_{\gamma\in[-1/2, 1/2]}\phi_j(\gamma)$ for the case where $\hat\nu_j<-1/2$. 
    
    Repeating for every coordinate $1\leq j\leq m$ and using the elementwise non-negativity of $c^\alpha$, we have that
    \begin{equation}\label{eq:nu_is_optimal}
        \nu=\argmin_{\mu\in[-1/2,1/2]^m}\left\{\sum_{j=1}^mc^\alpha_j\phi_j(\mu_j)=\DHa(\mu\|\hat\nu)\right\}.
    \end{equation}
    Therefore, $\operatorname{clip}(\hat\nu, -1/2, 1/2)$ can be understood as a Bregman projection onto the subset $[-1/2, 1/2]^m$. Now, using the three-points identity for Bregman divergences, we have
    \begin{align*}
        \DHa(\theta\|\nu)=\DHa(\theta\|\hat\nu)-\DHa(\nu\|\hat\nu) -\inner{\nabla H_c^\alpha(\nu)-\nabla H_c^\alpha(\hat\nu)}{\theta-\nu}\\
        =\DHa(\theta\|\hat\nu)-\DHa(\nu\|\hat\nu) -\inner{\nabla_\nu \DHa(\nu\|\hat\nu)}{\theta-\nu}.
    \end{align*}
    By~\eqref{eq:constrained_opt} and~\eqref{eq:nu_is_optimal},  $\inner{\nabla_\nu \DHa(\nu\|\hat\nu)}{\theta-\nu}\geq 0$ for all $\theta\in[-1/2, 1/2]^m$. Then, we obtain
    \begin{align*}
        \DHa(\theta\|\nu)\leq\DHa(\theta\|\hat\nu)-\DHa(\nu\|\hat\nu)
        \leq \DHa(\theta\|\hat\nu),
    \end{align*}
    where the second inequality follows from the nonnegativity of Bregman divergences for convex functions. We therefore conclude the proof.
\end{proof}

Next, we note a simple identity following from the definition of $p^{\eta_t}(\theta)$ and $Z^t$.
\begin{lemma}\label{lem:pd_obj_relation}
    Let $\theta\in [-1,1]^m$ and $t\geq 1$ and define $\tilde X=\diag{r}p^{\eta_t}(\theta)$. Then, we have the following identity
    \begin{equation}\label{eq:pd_obj_relation}
        -\inner{r}{\log Z^t(\theta)}-t\inner{\col(\tilde X)}{\theta}=\frac{t}{2}\inner{C}{\tilde X}-H(\tilde X)+H(r).
    \end{equation}
\end{lemma}
\begin{proof}
    First, note that $Z^t(\theta)_i$ is the normalization constant for the $i^{th}$ row of $p^{\eta_t}(\theta)$. Then, from~\eqref{def:pd_map_appdx}, we have
    \begin{align*}
       -H(\tilde X)=&\inner{\diag{r}p^{\eta_t}(\theta)}{\log \diag{r} p^{\eta_t}(\theta)}
       =\inner{\diag{r}p^{\eta_t}(\theta)}{\log p^{\eta_t}(\theta)}+\inner{r}{\log r}\\
       \stackrel{\eqref{def:pd_map_appdx}}=&-t\inner{\diag{r}p^{\eta_t}(\theta)}{\frac{C}{2}+ \one\theta^\top}-\inner{\diag{r}p^{\eta_t}(\theta)}{\log Z^t(\theta)}-H(r)\\
       =&-\frac{t}{2}\inner{\tilde X}{C}-t\inner{\col(\tilde X)}{\theta}-\inner{r}{\log Z^t(\theta)}-H(r).
    \end{align*}
    Rearranging yields the result.
\end{proof}

We start by stating consequences of the optimality conditions in~\eqref{def:bar_nu}-\eqref{def:theta}.
\begin{lemma}\label{lem:logterms}
    For all $t\geq 0$, the following equalities hold
    \begin{itemize}
        \item[a)]
        \begin{equation}\label{eq:logterm_primal}
            \log X^{t+1}-\log X^{t}=-\frac{1}{2}C-\one(\bar\theta^{t+1})^\top-\log Z^{t+1}(\nu^{t+1})\one^\top+\log Z^{t}(\nu^{t})\one^\top;
        \end{equation}
        \item[b)]
        \begin{equation}\label{eq:opt_dual_mid}
            \nabla H_c^\alpha(\bar \theta^{t+1})-\nabla H_c^\alpha( \theta^t)=\col(X^{t})-c;
        \end{equation}
        \item[c)]
        \begin{equation}\label{eq:opt_dual_hat}
            \nabla H_c^\alpha(\hat \theta^{t+1})-\nabla H_c^\alpha( \theta^t)=\col(\bar X^{t+1})-c.
        \end{equation}
    \end{itemize}
\end{lemma}
\begin{proof}
    a) Using the definition of the dual-to-primal map~\eqref{def:pd_map_appdx} and canceling the common $\log (r\one^\top)$ terms gives 
    \begin{align*}
      \log X^{t+1}-\log X^{t}\stackrel{\eqref{def:primal_subs}}=& \log p^{\eta_{t+1}}(\nu^{t+1})-\log p^{t}(\nu^{t})\\
      \stackrel{\eqref{def:pd_map_appdx}}=&-(t+1)\left(\frac{1}{2}C+\one(\nu^{t+1})^\top\right)+t\left(\frac{1}{2}C+\one(\nu^{t})^\top\right)\\&-\log Z^{t+1}(\nu^{t+1})\one^\top+\log Z^{t}(\nu^{t})\one^\top\\
      =&-\frac{1}{2}C+\one\left(t\nu^t-(t+1)\nu^{t+1}\right)^\top-\log Z^{t+1}(\nu^{t+1})\one^\top+\log Z^{t}(\nu^{t})\one^\top\\
      \stackrel{\eqref{def:nu}}=&-\frac{1}{2}C-\one(\bar\theta^{t+1})^\top-\log Z^{t+1}(\nu^{t+1})\one^\top+\log Z^{t}(\nu^{t})\one^\top,
    \end{align*}
    where the final line follows from the $\nu^{t+1}$ update.

    b) The optimality conditions for problem~\eqref{def:bar_theta} give 
    \begin{equation}
    0=\col(X^t)-c -\nabla H_c^\alpha(\bar\theta^{t+1})+\nabla H_c^\alpha(\theta^t).
    \end{equation}
    Rearranging gives the result. Statement c) follows from the same logic, noting that the optimality conditions for problem~\eqref{def:theta_hat} give
    \begin{equation}
        0=\col(\bar X^{t+1})-c -\nabla H_c^\alpha(\hat\theta^{t+1})+\nabla H_c^\alpha(\theta^t).
    \end{equation}
\end{proof}

Using the primal optimality conditions in Lemma~\ref{lem:logterms}, we can show the following identity for differences of KL-divergence terms.

\begin{lemma}\label{lem:kl_diff}
    For all $t\geq 0$, the following identities hold
    \begin{align}
       \nonumber \KL(\bar X^{t+1}\|X^{t})-\KL(\bar X^{t+1}\|X^{t+1})=&
        -\inner{\col(\bar X^{t+1})}{\bar\theta^{t+1}}
        -\frac{1}{2}\inner{C}{\bar X^{t+1}}\\
        &-\inner{r}{\log Z^{t+1}(\nu^{t+1})}+\inner{r}{\log Z^{t}(\nu^{t})}.\label{eq:kl_diff}
    \end{align}
\end{lemma}
\begin{proof}
    Note that by the definition of the KL-divergence $\KL$, we have
    \begin{align*}
        \KL(\bar X^{t+1}\|X^{t})-&\KL(\bar X^{t+1}\|X^{t+1})=\inner{\bar X^{t+1}}{\log X^{t+1}-\log X^{t}}\\
        \stackrel{\eqref{eq:logterm_primal}}=&\inner{\bar X^{t+1}}{-\frac{1}{2}C-\one(\bar\theta^{t+1})^\top-\log Z^{t+1}(\nu^{t+1})\one^\top+\log Z^{t}(\nu^{t})\one^\top},
    \end{align*}
    where the second line follows by Lemma~\ref{lem:logterms}. Rearranging and using $\bar X^{t+1}\one=r$ and $ (\bar X^{t+1})^\top\one=\col(\bar X^{t+1})$ gives the result.
\end{proof}


\begin{lemma}\label{lem:Dha_diff}
    For all iterations $t\geq 0$ and any $\theta\in[-1/2,1/2]^m$, we have the following identity
    \begin{align}
        \nonumber\DHa(\theta\|\hat\theta^{t+1})-\DHa(\theta\|\theta^t)+\DHa(\bar\theta^{t+1}\|\theta^t)-&\DHa(\bar\theta^{t+1}\|\hat\theta^{t+1})
        \\
        =&\inner{\bar\theta^{t+1} - \theta}{\col(\bar X^{t+1})-c}.\label{eq:dual_div_diff}
    \end{align}
\end{lemma}
\begin{proof}
    Using the definition of $\DHa$ in~\eqref{def:DHa} and expanding terms, we obtain
    \begin{align*}
       \DHa(\theta&\|\hat\theta^{t+1}) -\DHa(\theta\|\theta^t)+\DHa(\bar\theta^{t+1}\|\theta^t)-\DHa(\bar\theta^{t+1}\|\hat\theta^{t+1})\\
        \stackrel{\eqref{def:DHa}}=&\inner{\frac{c^\alpha}{2}}{\log\frac{1-\theta^2}{1-(\hat\theta^{t+1})^2}}-\inner{\frac{c^\alpha}{2}}{\log\frac{1-\theta^2}{1-(\theta^t)^2}} +\inner{\theta}{\nabla H_c^\alpha(\theta^{t})-\nabla H_c^\alpha(\hat\theta^{t+1})}\\
        &+\inner{\frac{c^\alpha}{2}}{\log\frac{1-(\bar\theta^{t+1})^2}{1-(\theta^t)^2}}-\inner{\frac{c^\alpha}{2}}{\log\frac{1-(\bar\theta^{t+1})^2}{1-(\hat\theta^{t+1})^2}} +\inner{\bar\theta^{t+1}}{\nabla H_c^\alpha(\hat\theta^{t+1})- \nabla H_c^\alpha(\theta^{t})}\\
        =&\inner{\bar\theta^{t+1} - \theta}{\nabla H_c^\alpha(\hat\theta^{t+1})- \nabla H_c^\alpha(\theta^{t})}
        \stackrel{\eqref{eq:opt_dual_hat}}=\inner{\bar\theta^{t+1} - \theta}{\col(\bar X^{t+1})-c},
    \end{align*}
    where the second equality follows by algebra and the final line follows by Lemma~\ref{lem:logterms}(c).
\end{proof}

With the primary identities/inequalities established, we now prove a single-step bound for Algorithm~\ref{alg:lamp_appdx}.
\begin{lemma}\label{lem:single_step}
 Let $X^*\in\coup(r,c)$ be an optimizer of~\eqref{prob:primal_ot}. Then, for all $t\geq 0$ and any $\theta\in[-1/2,1/2]^m$, we have the inequality
    \begin{align}
        \nonumber\KL(X^*\|X^t)&-\KL(X^*\|X^{t+1})\geq\frac{1}{2}\inner{\bar X^{t+1}-X^*}{C}+\inner{\col(\bar X^{t+1})-c}{\theta}\\
        \nonumber&-\DHa(\theta\|\theta^{t})+\DHa(\theta\|\hat\theta^{t+1})+\DHa(\bar\theta^{t+1}\|\theta^{t})-\DHa(\bar\theta^{t+1}\|\hat\theta^{t+1}) \\
        &+ \KL(\bar X^{t+1}\|X^t)-\KL(\bar X^{t+1}\|X^{t+1}).\label{ineq:single_step}
        \end{align}
\end{lemma}
\begin{proof}
First, expanding the definition of the KL-divergence gives
    \begin{align*}
        \KL(X^*&\|X^t)-\KL(X^*\|X^{t+1}) - \KL(\bar X^{t+1}\|X^t) + \KL(\bar X^{t+1}\|X^{t+1}) \\
        =& \inner{\bar X^{t+1}-X^*}{\log X^t-\log X^{t+1}}\\
        \stackrel{\eqref{eq:logterm_primal}}=&\frac{1}{2}\inner{\bar X^{t+1}-X^*}{C}
        +\inner{\col(\bar X^{t+1})-c}{\bar\theta^{t+1}} \\
        &+\inner{\bar X^{t+1}-X^*}{(\log Z^{t+1}(\nu^{t+1})-\log Z^{t}(\nu^t))\one^\top}\\
        =&\frac{1}{2}\inner{\bar X^{t+1}-X^*}{C}+\inner{\col(\bar X^{t+1})-c}{\hat\theta^{t+1}}
        +\inner{\col(\bar X^{t+1})-c}{\bar\theta^{t+1}-\hat\theta^{t+1}}
        ,
        \end{align*}
        where the second line follows by Lemma~\ref{lem:logterms} and the third line by algebra and the fact that $(\bar X^{t+1}-X^*)\one=0$ by~\eqref{eq:row_feasible} and the row feasibility of $X^*$. Adding and subtracting $\col(X^t)$ in the second inner product term then yields
        \begin{align}
        \nonumber\KL(X^*\|X^t)-&\KL(X^*\|X^{t+1})=\frac{1}{2}\inner{\bar X^{t+1}-X^*}{C}+\inner{\col(\bar X^{t+1})-c}{\hat\theta^{t+1}}\\
        \nonumber&+\inner{\col(\bar X^{t+1})-\col(X^t)}{\bar\theta^{t+1}-\hat\theta^{t+1}}
        +\inner{\col( X^t)-c}{\bar\theta^{t+1}-\hat\theta^{t+1}}\\
        &
        +\KL(\bar X^{t+1}\|X^t)-\KL(\bar X^{t+1}\|X^{t+1}).\label{eq:kl_difference_e7}
    \end{align}
    Now we can begin to utilize the dual optimality conditions. First, from~\eqref{def:bar_theta} and~\eqref{def:theta_hat}, we have
    \begin{align}
        \nonumber\inner{\col(\bar X^{t+1})-\col(X^t)}{\bar\theta^{t+1}-\hat\theta^{t+1}}\stackrel{\eqref{eq:opt_dual_mid}\eqref{eq:opt_dual_hat}}=-\inner{\nabla H_c^\alpha(\hat\theta^{t+1})-\nabla H_c^\alpha(\bar\theta^{t+1})}{\hat\theta^{t+1}-\bar\theta^{t+1}}\\
        =-\DHa(\hat\theta^{t+1}\|\bar\theta^{t+1})-\DHa(\bar\theta^{t+1}\|\hat\theta^{t+1})\label{eq:dual_div_difference},
    \end{align}
    where the second equality can be verified by simple algebra.
    Next, we have
    \begin{align}
        \nonumber-\inner{\col(X^t)-c}{\hat\theta^{t+1}-\bar\theta^{t+1}}\stackrel{\eqref{eq:opt_dual_mid}}=-\inner{\nabla H_c^\alpha(\bar\theta^{t+1})-\nabla H_c^\alpha(\theta^{t})}{\hat\theta^{t+1}-\bar\theta^{t+1}}\\
        =\DHa(\hat\theta^{t+1}\|\bar\theta^{t+1})-\DHa(\hat\theta^{t+1}\|\theta^t)+\DHa(\bar\theta^{t+1}\|\theta^t),\label{eq:dual_three_point}
    \end{align}
    Finally, using the three-point inequality from~\cite[Lemma 3.2]{chen1993convergence} applied to the negation of~\eqref{def:theta_hat} using the definitions in~\eqref{def:primal_subs}, we obtain 
    \begin{align}
        \nonumber\inner{\col(\bar X^{t+1})-c}{\hat\theta^{t+1}}\geq& \inner{\col(\bar X^{t+1})-c}{\theta} - \DHa(\theta\|\theta^t)\\
        &+\DHa(\hat\theta^{t+1}\|\theta^t)+\DHa(\theta\|\hat\theta^{t+1}).\label{ineq:hat_optimality}    \end{align}

    Combining~\eqref{eq:dual_div_difference},~\eqref{eq:dual_three_point}, and~\eqref{ineq:hat_optimality} with~\eqref{eq:kl_difference_e7}, we obtain the result~\eqref{ineq:single_step}.
\end{proof}

We now prove Proposition~\ref{prop:infeas_bound}. For expository purposes, we state the two conclusions (a) and (b) as separate lemmas, each following from Lemma~\ref{lem:single_step}.
\begin{lemma}\label{lem:last_iterate}
    For all $t\geq 1$, we have the inequality
    
    \begin{equation}
        \inner{\widetilde{X}^t}{C}-\min_{X\in\coup(r,c)}\inner{X}{C}\leq 4\|C\|_\infty\|\col(X^t)-c\|_1 + \frac{2\|C\|_\infty\log m}{t},
    \end{equation}
    where $\widetilde{X}^t=\operatorname{Round}(X^t, r, c)$.
\end{lemma}
\begin{proof}        
        Starting from Lemma~\ref{lem:single_step} and applying Lemmas~\ref{lem:kl_diff} and~\ref{lem:Dha_diff} and algebra, we obtain for any $\theta\in[-1/2, 1/2]^m$
        \begin{align*}
        \KL(X^*\|X^t)&-\KL(X^*\|X^{t+1})\stackrel{\eqref{ineq:single_step}}\geq\frac{1}{2}\inner{\bar X^{t+1}-X^*}{C}+\inner{\col(\bar X^{t+1})-c}{\theta}\\
        &-\DHa(\theta\|\theta^{t})+\DHa(\theta\|\hat\theta^{t+1}) +\DHa(\bar\theta^{t+1}\|\theta^{t})-\DHa(\bar\theta^{t+1}\|\hat\theta^{t+1}) \\
        &+ \KL(\bar X^{t+1}\|X^t)-\KL(\bar X^{t+1}\|X^{t+1})\\
        \stackrel{\eqref{eq:kl_diff}\eqref{eq:dual_div_diff}}=&-\frac{1}{2}\inner{X^*}{C}-\inner{\bar\theta^{t+1}}{c} + \inner{r}{\log Z^{t}(\nu^{t}) - \log Z^{t+1}(\nu^{t+1})}\\
        \stackrel{\eqref{def:nu}}=&-\frac{1}{2}\inner{X^*}{C}+\inner{c}{t\nu^t-(t+1)\nu^{t+1}} + \inner{r}{\log Z^{t}(\nu^{t}) - \log Z^{t+1}(\nu^{t+1})},
    \end{align*}
    where the last two equalities follow from algebra.

    Summing from $0$ to $t-1$ for $t\geq 1$, we have
    \begin{align*}
        \KL(X^*\|X^0)&-\KL(X^*\|X^{t})\geq-\frac{t}{2}\inner{X^*}{C}+t\inner{\col(X^t)-c}{\nu^{t}}\\
        &-t\inner{\col(X^t)}{\nu^{t}}-\inner{r}{\log Z^{t}(\nu^{t})} +\inner{r}{\log Z^{0}(\nu^{0})}\\
        \stackrel{\eqref{def:primal_subs}\eqref{eq:pd_obj_relation}}=&\frac{t}{2}\inner{X^t-X^*}{C}+t\inner{\col(X^t)-c}{\nu^{t}}-H(X^t)+H(r)+\log m\\
        \geq& \frac{t}{2}\inner{X^t-X^*}{C}+t\inner{\col(X^t)-c}{\nu^{t}},
        \end{align*}
    where the equality follows from Lemma~\ref{lem:pd_obj_relation} with $(\theta,\tilde X)=(\nu^t,X^t)$, and the facts that $\eta_0^{-1}=0$ and
    \[
    \inner{r}{\log Z^0(\nu^0)}=\sum_{i=1}^nr_i\LSE(\zero_m) = \sum_{i=1}^nr_i \log m=\log m.
    \]
    The final line follows from
    \[
    -H(X^t)+H(r)\geq-H(X^t)=-\sum_{i=1}^nr_iH(p_{i:})\geq -\sum_{i=1}^nr_i\log m = -\log m,\]
    where we use the fact that each row $p_{i:}\in\Delta^m$.
    
    Rearranging, discarding the negative $\KL(X^*\|X^{t})$ term and dividing both sides by $t/2$ gives,
    \begin{align}
    \nonumber\inner{X^t-X^*}{C}&\leq 2\inner{-\nu^t}{\col(X^t)-c}+\frac{2\KL(X^*\|X^0)}{t}\\
    &\leq 2\|\col(X^t)-c\|_1+\frac{2\KL(X^*\|X^0)}{t}
    \leq 2\|\col(X^t)-c\|_1+\frac{2\log m}{t},\label{ineq:ub_telescope_last_iter}
    \end{align}
    where the second inequality follows from Hölder's inequality and using $\|\nu^t\|_\infty\leq 1$ and the final line from the choice $X^0=\diag{r}(1/m)^{n\times m}$. 
    
    The lower bound follows from defining $\widetilde{X}^t=\operatorname{Round}(X^t, r, c)$ and applying Lemma~\ref{lem:jason_rounding} to obtain
    \begin{align}
        \nonumber\inner{X^t-X^*}{C}=&\inner{\widetilde{X}^t-X^*}{C}+\inner{X^t-\widetilde{X}^t}{C}
        \geq \inner{\widetilde{X}^t-X^*}{C}-\|C\|_\infty\|X^t-\widetilde{X}^t\|_1\\
        \label{ineq:lb_tilde_X}\stackrel{\eqref{ineq:l1_normbound_jason}}\geq&\inner{\widetilde{X}^t-X^*}{C}-2\|\col(X^t)-c\|_1,
    \end{align}
    where the first inequality follows from Hölder's inequality and the second from Lemma~\ref{lem:jason_rounding}. Combining~\eqref{ineq:ub_telescope_last_iter} with~\eqref{ineq:lb_tilde_X}  gives the result for $\|C\|_\infty=1$. If $\|C\|_\infty\neq 1$, then applying the analysis to the normalized cost $C/\|C\|_\infty$  and multiplying both sides by $\|C\|_\infty$ gives the claimed inequality.
\end{proof}

\begin{lemma}\label{lem:summability}
    Let $X^*\in\coup(r,c)$ be a primal optimizer of~\eqref{prob:primal_ot}. Then, defining
    \begin{equation}\label{def:sum_condition_2}
    D_t:=\sum_{s=0}^{t-1}\DHa(\bar\theta^{s+1}\|\hat\theta^{s+1})- \DHa(\bar\theta^{s+1}\|\theta^{s})+\KL(\bar X^{s+1}\|X^{s+1})- \KL(\bar X^{s+1}\|X^{s}),\end{equation}
    then for all $t\geq 1$
    \begin{equation}\label{ineq:spp_midpoint_conv_appdx}
        \inner{C}{\text{Round}(\hat X^t,r,c)-X^*}\leq\frac{2\|C\|_\infty[\log m+(1+\alpha)\log 2+D_t]}{t},
    \end{equation}
    where $\hat X^t$ is the ergodic iterate
    \begin{equation}\label{def:xhat}
        \hat X^t=\frac{1}{t}\sum_{i=1}^t \bar X^i.
    \end{equation}
\end{lemma}
\begin{proof}        
        Assume that $\|C\|_\infty=1$. Lemma~\ref{lem:single_step} implies that for any $\theta\in[-1/2, 1/2]^m$, we have the bound 
        \begin{align*}
        \frac{1}{2}[\inner{C}{\bar X^{t+1}-X^*}+&2\inner{\theta}{\col(\bar X^{t+1})-c}] \stackrel{\eqref{ineq:single_step}}\leq \KL(X^*\|X^t)-\KL(X^*\|X^{t+1})\\
        &+\DHa(\bar\theta^{t+1}\|\hat\theta^{t+1})- \DHa(\bar\theta^{t+1}\|\theta^{t})+\DHa(\theta\|\theta^{t})-\DHa(\theta\|\hat\theta^{t+1})\\
        &+\KL(\bar X^{t+1}\|X^{t+1})- \KL(\bar X^{t+1}\|X^{t})\\
        &\stackrel{\eqref{def:theta}\eqref{ineq:clipping_divergence}}\leq \KL(X^*\|X^t)-\KL(X^*\|X^{t+1})\\
        &+\DHa(\bar\theta^{t+1}\|\hat\theta^{t+1})- \DHa(\bar\theta^{t+1}\|\theta^{t})+\DHa(\theta\|\theta^{t})-\DHa(\theta\|\theta^{t+1})\\
        &+\KL(\bar X^{t+1}\|X^{t+1})- \KL(\bar X^{t+1}\|X^{t})
    \end{align*}
    where the second inequality uses Lemma~\ref{lem:clipping_divergence} and the assumption that $\theta\in[-1/2,1/2]^m$.
    Summing from $0$ to $t-1$ and using the definition of $D_t$ in~\eqref{def:sum_condition_2} gives
    \begin{align*}
        &\KL(X^*\|X^0)-\KL(X^*\|X^{t})+\DHa(\theta\|\theta^0)-\DHa(\theta\|\theta^t)+D_t\\
        \geq & \frac{1}{2}\sum_{s=1}^{t}\inner{C}{\bar X^s-X^*}+2\inner{\theta}{\col(\bar X^s)-c}
        =\frac{t}{2}\left[\inner{C}{\hat X^t-X^*}+2\inner{\theta}{\col(\hat X^t)-c}\right],
    \end{align*}
    where the equality follows from the definition of $\hat X^t$ in~\eqref{def:xhat}.
    Note that
    \begin{align}
        \nonumber\DHa(\theta\|\theta^0)&\stackrel{\eqref{def:DHa_explicit}}=\inner{c^\alpha}{\frac{\theta+1}{2}\log(\theta+1) + \frac{1-\theta}{2}\log(1-\theta)}\\
        \label{ineq:dual_div_bound}&\leq \inner{c^\alpha}{\frac{\theta+1}{2}\log2 + \frac{1-\theta}{2}\log2}=(1+\alpha)\log 2
    \end{align}
    since $\theta^0=\zero_m$ and $\theta\in[-1/2,1/2]^m\subset[-1,1]^m$. Applying this bound, discarding negative terms, and taking the maximum over $\theta$ on both sides, we obtain
    \begin{align}
        \nonumber\inner{C}{\hat X^t-X^*}+\|\col(\hat X^t)-c\|_1=&\max_{\theta\in[-1/2, 1/2]^m}\left\{\inner{C}{\hat X^t-X^*}+2\inner{\theta}{\col(\hat X^t)-c}\right\}\\
        \leq& \frac{2[\KL(X^*\|X^0)+(1+\alpha)\log 2 + D_t]}{t}.\label{ineq:sup_bound}
    \end{align}
    Finally, applying Lemma~\ref{lem:jason_rounding} and using the fact that $\row(\hat X^t)=r$, we obtain
    \begin{align}
        \nonumber\inner{C}{\text{Round}(\hat X^t,r,c)-X^*}=& \inner{C}{\hat X^t-X^*}+\inner{C}{\text{Round}(\hat X^t,r,c) - \hat X^t}\\
        \nonumber\leq& \inner{C}{\hat X^t-X^*}+\|\hat X^t-\text{Round}(\hat X^t,r,c)\|_1\\
         \stackrel{\eqref{ineq:l1_normbound_jason}}\leq&\inner{C}{\hat X^t-X^*}+\|\col(\hat X^t)-c\|_1,\label{ineq:ergodic_round_bound}
    \end{align}
    where the first inequality follows by H\"{o}lder's inequality and the assumption that $\|C\|_\infty=1$. 
    The result for $\|C\|_\infty=1$ then directly follows by combining~\eqref{ineq:sup_bound} and~\eqref{ineq:ergodic_round_bound} and using the fact that $\KL(X^*\|X^0)\leq\log m$. As in Lemma~\ref{lem:last_iterate}, applying the analysis to the normalized cost and multiplying both sides by $\|C\|_\infty$ gives the result.
\end{proof}
Proposition~\ref{prop:infeas_bound} then follows from combining Lemmas~\ref{lem:last_iterate} and~\ref{lem:summability}. Lemma~\ref{lem:last_iterate} provides last-iterate sub-optimality certificates, but requires convergence in the column infeasibility $\|\col(X^t)-c\|_1$ to imply convergence in objective value. As observed in Section~\ref{sec:numerical} (see Fig.~\ref{fig:metric_comp}), LAMP rapidly finds feasible solutions, with convergence rates then limited by the sublinear term.

Lemma~\ref{lem:summability} further shows that the rounded ergodic iterate $\text{Round}(\hat X^t, r, c)$ produces an $\varepsilon$-solution to~\eqref{prob:primal_ot} in $\mathcal{O}(\varepsilon^{-1}(\log m + D_t))$ iterations.  Therefore, proving an $o(t)$ bound on $D_t$ would then imply the ergodic convergence of LAMP, with non-asymptotic complexity dependent on the specific bound. Note that $D_t$ is computable, and therefore can be tracked along iterate trajectories to measure ergodic stability. Fig.~\ref{fig:Dt_experiment} shows the empirical behavior of $\{D_t\}_{t\geq 0}$ for cell similarity (Fig.~\ref{fig:cell_sum_Dt}) and DOTmark (Fig.~\ref{fig:dotmark_Dt}). For both problem sets, the $D_t$ rapidly becomes negative, then approaches 0 from below. While we observe some nonmonotonicity, the behavior is relatively benign. For the cell similarity problems, the choice of metric has little effect on the behavior of the $D_t$ curve. In contrast, the behavior of $D_t$ is more metric dependent for the DOTmark instances, as shown in in Fig.~\ref{fig:dotmark_Dt}. Each curve appears to converge to zero, but the convergence rate appears to decrease as the $\|C\|_\infty$ value increases. In the case of 2-dimensional DOTmark problems, we have $\|C\|_\infty=\sqrt{n}-1$ for $\ell_\infty$, $\|C\|_\infty=2(\sqrt{n}-1)$ for $\ell_1$, and $\|C\|_\infty=2(\sqrt{n}-1)^2$ for $\ell_2^2$.

\begin{figure}[h]
    \begin{subfigure}[t]{0.47\linewidth}
       \centering 
       \includegraphics[width=\linewidth]{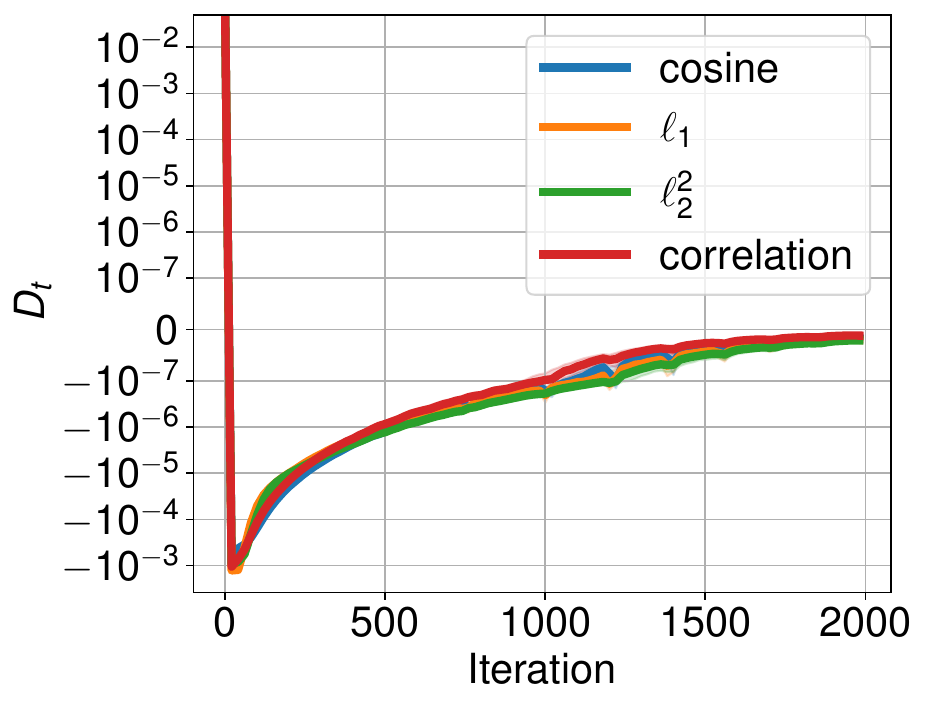}
    \subcaption{Behavior of the summed sequence $D_t$ for the cell similarity OT problems ($n\in\{256,576, 1024, 1600\}$ and $n=m$).}
    \label{fig:cell_sum_Dt}
    \end{subfigure}
    \hfill
    \begin{subfigure}[t]{0.47\linewidth}
       \centering 
       \includegraphics[width=\linewidth]{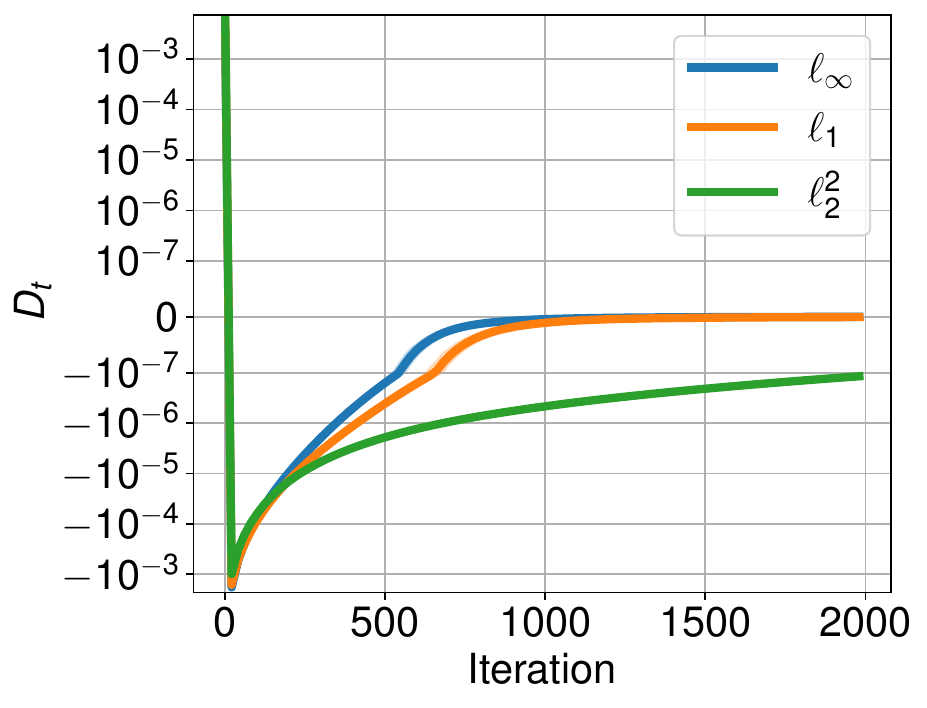}
    \subcaption{Behavior of the summed sequence $D_t$ for DOTmark OT problems ($n\in\{256,576, 1024, 1600\}$ and $n=m$).}
    \label{fig:dotmark_Dt}
    \end{subfigure}
    \caption{Behavior of the summed sequence $D_t$ in OT benchmark problems from the cell similarity database [left] and DOTmark [right].}
    \label{fig:Dt_experiment}
\end{figure}

\section{Numerical and Experimental Details}\label{appdx:numerical}
In this section we provide further details on our GPU-accelerated implementation of LAMP and the numerical experiments in Sections~\ref{sec:lamp},~\ref{sec:implementation}, and~\ref{sec:numerical}.

\subsection{Fused/Warp-tiled CUDA Reductions}
For both LAMP and Sinkhorn, we perform the $\max$ and $\LSE$ reductions in a single loop-fused and warp-tiled kernel $\texttt{warp\_lse\_reduction}$. 

Loop fusion replaces the separate reduction loops with a single loop, which performs an on-the-fly variant of the \texttt{LogSumExp} trick. Consider the generic problem of computing $\LSE(x_{i:})$ over $x\in\R^{n\times m}$ (i.e., a row-wise LSE reduction). The fused loop maintains two sequences, $m^k_i$ and $s^k_i$, for storing the maximum value and accumulated exponential sums for row $i$ respectively. We initialize $m^0_i=-\infty$,\footnote{AKA $-\texttt{DBL\_MAX}$} $s^0_i=0$. For $1\leq j\leq m$, if $x_{ij}\geq m_i^{j-1}$, then we perform the update
\[m_i^j=x_{ij},\quad s^j_i=1+(s^{j-1}_i)^{-m_{i}^j+m_i^{j-1}}.\] 
Otherwise, we simply set $m_i^{j}=m_i^{j-1}$ and update
\[s^j_i=s^{j-1}_i+\exp[x_{ij}-m_i^{j-1}].\]

Warp-tiling splits the computation into a two-loop structure. The outer loop iterates $n/n_{\textrm{Warps}}$ times and performs the row-wise $\LSE$ and $\max$ reductions for $n_{\textrm{Warps}}$ rows simultaneously. The reduction operations are warp-tiled, meaning that each warp computes $32$ terms of each reduction sum in parallel. After iterating through all $m$ elements, the warp threads perform synchronized reductions/broadcasting using \texttt{\_\_shfl\_down\_sync()} and \texttt{\_\_shfl\_sync()}  primitives. Since Julia uses column-major storage, we store a transposed copy of the $C$ matrix (for non-kernelized costs) to ensure that memory accesses are contiguous within each warp.

\subsection{Experimental Details}
We now provide additional details on problem selection, termination criteria, and problem parameters.

We compare LAMP to a log-domain stabilized implementation of Sinkhorn's algorithm~\cite{sinkhornConcerningNonnegativeMatrices1967,altschulerNearlinearTimeApproximation2017,schmitzerStabilizedSparseScaling2019a}, Dual Extrapolation~\cite{jambulapatiDirectTildelbraceOrbrace2019a}, APDAMD~\cite{linEfficiencyEntropicRegularized2022}, Accelerated Sinkhorn~\cite{linEfficiencyEntropicRegularized2022}, and HPD~\cite{chambolleAcceleratedBregmanPrimalDual2022}. We implement each solver in Julia with GPU-accelerated operations using the CUDA.jl package. The implementations, data, and experiment/plotting code are available in a public repository.\footnote{\href{https://github.com/mxburns2022/CuLAMP.jl}{https://\textbf{github}.com/mxburns2022/CuLAMP.jl}}

Each solver was run with a maximum of $T$ iterations and a wall clock limit of $S$ seconds. The values of $T$ and $S$ varied by experiment, and are given for each figure below. Unless otherwise specified, if a solver reached $10^{-10}$ primal-dual gap (checked every 25-200 iterations, depending on the experiment), then the solver terminates early. For all EOT solvers, the dual function is given by~\eqref{def:dual_eot}. For LAMP the dual functions for $\eta =0$ and $\eta >0$ are given by~\eqref{def:dual_spp} and ~\eqref{def:dual_spp_e}, respectively. Dual extrapolation~\cite{jambulapatiDirectTildelbraceOrbrace2019a} used the $\eta =0$ dual function in~\eqref{def:dual_spp}. 

For all LAMP testing, we set $\beta\approx\log 3$ (specifically 1.1), $\alpha=0.01$, $\tau_1=\tau_2=1/(2\|C\|_\infty)$. Algorithm comparisons were performed on an HPC cluster running a Xeon Gold 6548Y+ CPU with an NVIDIA L40s GPU. 

To ensure that each marginal has full support, we first add a small $10^{-6}$ additive perturbation to the marginals and then renormalize.
 

\vspace{1em}\noindent\textbf{Details for Figure~\ref{fig:eot_broad}}
We choose five $32\times 32$ problems each from the DOTmark~\cite{schrieberDOTmarkBenchmarkDiscrete2017} classes ``ClassicImages'', ``GRFsmooth'', and ``GRFrough'' for a total of 15 problems. Each solver is given a 20 minute/$10^6$ iteration limit. Elapsed wall time, objective value, infeasibility, and primal/dual values are reported every 25 iterations. We set the hyperparameters for each algorithm using $\eta$ as described in each original reference.

\vspace{1em}\noindent\textbf{Details for Figure~\ref{fig:metric_comp}}
We choose ten $64\times 64$ image pairs each from the DOTmark~\cite{schrieberDOTmarkBenchmarkDiscrete2017} classes ``ClassicImages'', ``GRFsmooth'', and ``MicroscopyImages'' for a total of 30 problem instances. Furthermore, we run each problem using $\ell_2^2$, $\ell_1$, and $\ell_\infty$ ground costs. Kernel-based implementations of log-domain Sinkhorn and LAMP are each given a 10 minute/$10^6$ iteration limit. Unlike the non-kernelized Sinkhorn, we terminate the solver based on the condition $\|\col(X^t)-c\|_1+\|\row(X^t)-r\|_1\leq \varepsilon/2$ (which is $5\times 10^{-11}$) rather than a primal-dual gap. Elapsed wall time, objective value, infeasibility, and dual values are reported every 200 iterations for LAMP and Sinkhorn, while annealed Sinkhorn reports metrics after every call to the inner Sinkhorn routine. For annealed Sinkhorn, we use a multiplicative annealing schedule with $\eta_t=\max\{0.8^t\eta_i,\eta_f\}$ with $\eta_i=0.1$, $\eta_f=10^{-10}$ after brief tuning, and each subproblem is solved to $\varepsilon=10^{-10}$ accuracy.

\vspace{1em}\noindent\textbf{Details for Figure~\ref{fig:cell_similarity}}: We utilize the cell omics dataset collected by~\cite{liu2019deconvolution} (published under a CC-BY-4.0 license) and preprocessed by~\cite{huizingOptimalTransportImproves2022}. The preprocessed dataset used can be found in the publicly available repository~\href{https://github.com/cantinilab/OT-scOmics}{https://github.com/cantinilab/OT-scOmics} as \texttt{data/liu\_scatac\_preprocessed.csv.gz}. Each row of the dataset represents a genetic feature, with columns representing individual cells. Following~\cite{huizingOptimalTransportImproves2022}, we model the similarity between two cells by an optimal transport problem. The problem dimension $n=m$ is the number of genetic features. The marginals $r$ and $c$ are computed as the normalized feature vectors for each cell. For a $k$-cell database, costs are computed as 
\begin{equation}
    C_{ij}=\sum_{\ell=1}^kd(i_\ell,j_\ell),
\end{equation}
where $d(i_\ell,j_\ell)$ is a metric between feature $i$ of cell $\ell$ and feature $j$ of cell $\ell$. Therefore, the cost of computing each entry of $C$ scales $\mathcal{O}(k)$ (i.e., with the number of cells). We consider four different metrics $d$, 
\begin{enumerate}
    \item $\ell_1$ distance $d(x,y)=\|x-y\|_1$,
    \item $\ell_2^2$ distance $d(x,y)=\|x-y\|_2^2$,
    \item cosine similarity metric
    \begin{equation*}
        d(x,y)=1-\frac{\inner{x}{y}}{\|x\|_2\|y\|_2},
    \end{equation*}
    \item Pearson correlation metric
    \begin{equation*}
        d(x,y)=1-\frac{\inner{x-\bar{x}}{y-\bar{y}}}{\|x-\bar x\|_2\|y-\bar y\|_2},
    \end{equation*}
    where the sample means $\bar x$ and $\bar y$ as well as second moments are precomputed for each feature for efficiency.
\end{enumerate}

From the dataset, we select the $n=m=5000$ features with the highest variance and randomly subsample $k=20$ cells at random to form the dataset. We then randomly choose a pair of cells to create an OT problem instance. We repeat this process 10 times to generate 10 problem instances, which we then run with a 1 hour/100000 iteration timeout. The accuracy target was additionally set at $10^{-4}$, with algorithms terminating early if a $10^{-4}$-solution was detected using a primal-dual gap bound. Trajectories are averaged over the 10 problems. For Sinkhorn, we use $\eta=10^{-4}$. For annealed Sinkhorn, we use a multiplicative annealing schedule with $\eta_t=\max\{0.95^t\eta_i,\eta_f\}$ with $\eta_i=0.01$, $\eta_f=10^{-4}$ after brief tuning, and each subproblem is solved to $\varepsilon=10^{-4}$ accuracy.

\vspace{1em}\noindent\textbf{Details for Figure~\ref{fig:ctransfer}}
The original images were generated using Google Gemini 2 at a resolution of $1024\times 1024$. Note that our use of the images is in compliance with the Google Generative AI Terms of Service. The images were then rescaled using the Images.jl package and the RGB distributions were computed using a 1D kernel density estimator with $2^{\lceil{\log_2 n}\rceil}$ sample points as described in~\cite{pitie2007automated}. Denoting $\gamma_R$, $\gamma_G$, and $\gamma_B$ as the normalized KDE densities, we then compute each image marginal $\gamma$ as
\begin{equation*}
    \gamma_i \propto {\gamma_R(R_i) + \gamma_G(G_i) + \gamma_B(B_i)},
\end{equation*}
where $R_i$, $G_i$, and $B_i$ are respectively the RGB values of pixel $i$.
KDE densities are computed using the KernelDensity.jl package. Transfer maps are computed using $\ell_2^2$ ground costs. Upon terminating at iteration $T$, the color channels for the row-marginal transfer image is then computed as
\begin{align*}
    R^r_{\textrm{out}}&=p^{\eta_{T}}(\nu^{T})R^c,\quad
    G^r_\textrm{out}=p^{\eta_{T}}(\nu^{T})G^c,\quad
    B^r_\textrm{out}=p^{\eta_{T}}(\nu^{T})B^c, \\
    R^c_{\textrm{out}}&=(p^{\eta_{T}}(\nu^{T}))^\top R^r,\quad
    G^c_\textrm{out}=(p^{\eta_{T}}(\nu^{T}))^\top G^r,\quad
    B^c_\textrm{out}=(p^{\eta_{T}}(\nu^{T}))^\top B^r,
\end{align*}
where $R^r$, $G^r$, and $B^r$ are the input RGB values for the row-marginal image, and $R^c$, $G^c$, and $B^c$ are the input RGB values for the column-marginal image.

\vspace{1em}\noindent\textbf{Details for Figure~\ref{fig:dextrap}}
We choose five $32\times 32$ problems each from the DOTmark classes ``ClassicImages'', ``GRFsmooth'', and ``MicroscopyImages'' for a total of 15 problems. Prior to running the experiment, we performed brief parameter tuning to optimize the primal and dual stepsizes (both set to 1) and the number of inner iterations (set to 4) for Dual Extrapolation.

Both solvers were given a 10 minute/$10^6$ iteration limit. Since Dual Extrapolation has a two-loop structure with 4 inner iterations per outer iteration, we set a limit of $2.5\times 10^5$ total ``outer'' iterations. Elapsed wall time, objective value, infeasibility, and primal/dual values are reported every 200 outer iterations. 

\vspace{1em}\noindent\textbf{Details for Figure~\ref{fig:Dt_experiment}: } For Fig.~\ref{fig:cell_sum_Dt}, we construct 10 cell similarity problem instances each with $k=50$ cells, $n=m\in\{256, 576, 1024, 1600\}$ features, and costs computed using one of the four similarity measures described for Fig.~\ref{fig:cell_similarity} ($\ell_1$, $\ell_2^2$, cosine, Pearson correlation). The summand of $D_t$,
\[\DHa(\bar\theta^{s+1}\|\hat\theta^{s+1})- \DHa(\bar\theta^{s+1}\|\theta^{s})+\KL(\bar X^{s+1}\|X^{s+1})- \KL(\bar X^{s+1}\|X^{s}),\]
is logged at every iteration, with Fig.~\ref{fig:cell_sum_Dt} reporting the cumulative sum averaged over. The cost/marginal computations are identical to Fig.~\ref{fig:cell_similarity}. For Fig.~\ref{fig:dotmark_Dt}, we select 10 image pairs each from 3 DOTmark classes (``ClassicImages'', ``GRFsmooth'', and ``GRFrough'') with images sized to $r$ x $r$, where $r\in \{16, 24, 32, 40\}$, which gives $n=m\in \{256, 576, 1024, 1600\}$. Costs are computed using $\ell_1$, $\ell_2^2$, or $\ell_\infty$ ground costs. As in Fig.~\ref{fig:cell_sum_Dt}, the summand of $D_t$ is logged at every iteration, with the cumulative sum averaged over the problem instances for each cost metric. Each solver is given a 10 minute/100000 iteration time limit.

\vspace{1em}\noindent\textbf{Details for Table~\ref{tab:kernel_comparison}: }
All data for Table~\ref{tab:kernel_comparison} were collected on an RTX 4090 GPU with an Intel i9-13900k CPU and 64 GB of memory running NVIDIA driver 595.58.03. Kernel benchmarks were collected using Julia 1.12.6 with CUDA.jl version 5.11.0, and PyKeOps data was collected using Python 3.12.0, PyKeOps v2.3, and PyTorch 2.9.1 packaged with CUDA 12.6.

We generated synthetic data for each kernel of the specified size, then obtained the median and standard error kernel latency using the BenchmarkTools.jl package.

For PyKeOps, we instantiated each matrix/vector using PyTorch~\cite{paszkePyTorchImperativeStyle2019} CUDA utilities. Kernel operations were performed using the \texttt{Genred} function with the ``GPU'' backend. A sample execution time was measured using the \texttt{timeit} Python package with 500 samples. 50 samples for each size were collected, from which the median and standard error were computed. Kernel compilation was performed in the setup phase, and therefore did not contribute to the runtime.

\section{Review of PDMP Guarantees}\label{appdx:sketch}
In this section, we recall the convergence guarantees of~\cite{liFastComputationOptimal2025} for PDMP. We first provide the formal statement of Theorem~\ref{thm:li_informal} and provide a translation of our notation to that of~\cite{liFastComputationOptimal2025} to ease comparison.

Instead of maximizing over $[-1,1]^m$, the authors in~\cite{liFastComputationOptimal2025} dualize the $\ell_1$ norm of $x\in\R^m$ using the identity
\begin{equation}
    \|x\|_1=\max_{\mu\in\{\Delta^{2}\}^m}\inner{x}{\mu^+-\mu^-},
\end{equation}
where $\{\Delta^{2}\}^m$ is the set of $m\times 2$ matrices where each row lies in the 2-dimensional simplex. We then use the notation $\mu=(\mu^+,\mu^-)$ with $\mu^+,\mu^-\in[0,1]^m$ and so $\mu^+_j+\mu^-_j=1$ for all $j\in\{1\dots m\}$. It is simple to verify that $\mu^+_j-\mu^-_j\in[-1,1]$, hence we have the one-to-one translation
\begin{equation}\label{def:dual_translation}
    \theta(\mu)_j:=\mu^+_j-\mu^-_j=\theta_j
\end{equation}
for some $\theta\in[-1,1]^m$. Since $\mu^+_j+\mu^-_j=1$, we equivalently have
\begin{equation}\label{def:dual_translation2}
    \mu^+_j=\frac{1+\theta(\mu)_j}{2},\quad \mu^-_j=\frac{1-\theta(\mu)_j}{2}.
\end{equation}

Using the simplex dual notation, the authors of~\cite{liFastComputationOptimal2025} target the strongly-convex strongly-concave saddle-point problem
\begin{align}\label{prob:spp_reg}
    \nonumber\min_{p\in\{\Delta^m\}^n}\max_{\mu\in\{\Delta^{2}\}^m}\biggl\{\widetilde{K}^{\eta}(\diag{r}p,\mu):=&\frac{1}{2}\inner{C}{\diag{r}p} - \eta  H_r(p) \\
    &-\eta_\mu \widetilde H^\alpha_c(\mu)+ \|C\|_\infty\inner{\mu^+-\mu^-}{\col_r(p) -c}\biggr\},
\end{align}
where the authors also rescale by $1/2$ and we define the dual entropy
\[
\widetilde{H}^\alpha_c(\mu)=\sum_{j=1}^mc^\alpha_j(\mu^+_j\ln\mu^+_j+\mu^-_j\ln\mu^-_j),
\]
which is $\min_jc^\alpha_j$-strongly convex (hence $-\widetilde{H}^\alpha_c(\mu)$ is strongly concave) and can easily be shown to satisfy $\widetilde{H}^\alpha_c(\mu)=H_c^\alpha(\theta(\mu))$ using~\eqref{def:dual_translation2}.

Using the $\mu$ notation, Algorithm~\ref{alg:pdmp_mu} states PDMP more closely aligned with~\cite{liFastComputationOptimal2025}.
One can verify that the clipping step in~\eqref{eq:mu_clipping} is equivalent to the simplified, $\tanh$-based clipping step in~\eqref{eq:pdmp_theta_clipping}. 

We note that with $\eta_\mu>0$, the mirror map
\[
    \mathcal{M}_{\tau}^{\eta,\eta_\mu}(\theta^0;p)=\underset{\theta\in[-1,1]^m}\argmax\left\{\inner{\nabla_\theta \widetilde{K}^\eta(p,\mu(\theta^0))}{\theta} - \frac{1}{\tau}\D_{{H}^\alpha_c}(\theta\|\theta^0)\right\}
\]
has the closed-form solution
\[
    \mathcal{M}_{\tau}^{\eta,\eta_\mu}(\theta^0;p)_j=\tanh\biggl[\frac{2\tau\|C\|_\infty}{c^\alpha_j}(\col_r(p)_j-c_j)+\frac{1-\tau\eta_\mu}{2}\log\frac{1+\theta^0_j}{1-\theta^0_j}\biggr],
\]
where $\mu(\theta)$ is the transformation given in~\eqref{def:dual_translation2}.
The translation back to the $\mu$-parameterization can then be performed using \eqref{def:dual_translation2}. Therefore, a linear-space LAMP variant can also be defined for the doubly-regularized saddle-point problem~\eqref{prob:spp_reg}.

\begin{algorithm}\caption{Primal-Dual Mirror Prox}\label{alg:pdmp_mu}
    \begin{algorithmic}
        \REQUIRE $C\in\R_{+}^{n\times m}$, $r\in\Delta^n$, $c\in\Delta^m$, $\alpha>0$, $\tau_1>0$, $\tau_2>0$, $\eta \geq 0$, $\eta_\mu\geq 0$, set $p^0=(1/m)^{n\times m}$, $\mu^0=(1/2)^{m\times 2}$, $c^\alpha=c + \alpha m^{-1}\one_m$.
        \FOR{$t\geq 0$}
            \STATE \textbf{Step 1)} Compute the midpoints
            \begin{align*}
                \bar{p}^{t+1}&=\underset{p\in\{\Delta^m\}^n}\argmin\left\{\inner{\nabla_p \widetilde{K}^\eta(p^t,\mu^t)}{p} + \frac{1}{\tau_1}\KL(\diag{r}p\|\diag{r}p^t)\right\},\\
                \bar{\mu}^{t+1}&=\underset{\mu\in\{\Delta^2\}^m}\argmax\left\{\inner{\nabla_\mu \widetilde{K}^\eta(p^t,\mu^t)}{\mu} - \frac{1}{\tau_2}\D_{\widetilde{H}^\alpha_c(\mu)}(\mu\|\mu^t)\right\}.
            \end{align*}
            \vspace{-1em}
            \STATE \textbf{Step 2)} Compute the main sequence
            \begin{align*}
                {p}^{t+1}&=\underset{p\in\{\Delta^m\}^n}\argmin\left\{\inner{\nabla_p \widetilde{K}^\eta(p^t,\bar \mu^{t+1})}{p} + \frac{1}{\tau_1}\KL(\diag{r}p\|\diag{r}p^t)\right\},\\
                \hat{\mu}^{t+1}&=\underset{\mu\in\{\Delta^2\}^m}\argmax\left\{\inner{\nabla_\mu \widetilde{K}^\eta(\bar p^{t+1},\mu^t)}{\mu} - \frac{1}{\tau_2}\D_{\widetilde{H}^\alpha_c(\mu)}(\mu\|\mu^t)\right\}.
            \end{align*}
            \vspace{-0.5em}
            \STATE \textbf{Step 3)} Clip the dual variables
            \begin{equation}
                \mu^{t+1} = \operatorname{clip}\left(\hat\mu^{t+1}, \left[\frac{1}{1+e^\beta}, \frac{e^\beta}{1+e^\beta}\right]\right).\label{eq:mu_clipping}
            \end{equation}
        \ENDFOR
    \end{algorithmic}
\end{algorithm}
We now state formal version of Theorem~\ref{thm:li_informal}, which is the primary result of~\cite{liFastComputationOptimal2025}. Note that our parameters are slightly different from those given in~\cite[Equation 2.19]{liFastComputationOptimal2025} to account for the difference in formulation. The authors in~\cite{liFastComputationOptimal2025} pre-divide~\eqref{prob:spp_reg} by $2\|C\|_\infty$ and assume for their analysis that $\|C\|_\infty=1$. Accordingly, the $\varepsilon$ used to set parameters below scaled by $1/\|C\|_\infty$ and the primal/dual stepsizes $\tau_1$/$\tau_2$ are scaled by ${1}/{(2\|C\|_\infty)}$ compared with the statement in~\cite{liFastComputationOptimal2025}.

\begin{theorem}[{\cite[Theorem 2.2]{liFastComputationOptimal2025}}]\label{thm:main_convergence_li}
    Given $\varepsilon\in(0,\|C\|_\infty)$, set
    \begin{equation}
    \begin{gathered}
        \beta=C_1\log\frac{\|C\|_\infty m}{\varepsilon},\quad \hat\eta=\frac{C_2^2\varepsilon }{\sqrt{\beta}\|C\|_\infty \log m},\quad\tau_1=\frac{C_2}{2\|C\|_\infty\sqrt{\beta}},\quad \tau_2=\frac{15C_2\sqrt{\beta}}{2\|C\|_\infty},\quad \alpha=C_3,\\
        \eta =\frac{2\hat\eta}{\tau_{1}},\quad \eta_\mu=\frac{2\hat\eta}{\tau_{2}},
    \end{gathered}\label{def:li_parameters}
    \end{equation}
    where $C_1>0$, $C_2>0$, and $0<C_3\leq 1$ are some universal constants such that $C_1$ and $C_2\sqrt{C_1}$ are sufficiently large and $C_2$, $C_3$, and $C_2^2/C_3$ are sufficiently small, the number of iterations for Algorithm~\ref{alg:pdmp_mu} to obtain an $\varepsilon$-solution to \eqref{prob:primal_ot} is at most
    \begin{equation*}
        T=\mathcal{O}\left(\frac{1}{\hat\eta}\log\frac{m\|C\|_\infty}{\varepsilon}\right)=\mathcal{O}\left(\frac{\|C\|_\infty\log m}{\varepsilon}\log\frac{m\|C\|_\infty}{\varepsilon}\right).
    \end{equation*}
\end{theorem}
We refer interested readers to~\cite{liFastComputationOptimal2025} for the full proof of Theorem~\ref{thm:main_convergence_li}, which is quite technical and beyond the scope of this work.

\newpage

\end{document}